\documentclass[12pt]{iopart}
\usepackage{iopams}
\usepackage{setstack}
\usepackage{amsfonts}
\usepackage{amsthm}

\usepackage{amssymb}
\usepackage{graphicx}
\usepackage{epsfig}
\usepackage{psfrag}
\usepackage{hyperref}
\usepackage{color}

\usepackage[citation-order,msc-links]{amsrefs}
\newtheorem{theorem}{Theorem}

\newtheorem{corollary}[theorem]{Corollary}
\newtheorem{definition}[theorem]{Definition}
\newtheorem{lemma}[theorem]{Lemma}

\theoremstyle{remark}
\newtheorem{remark}[theorem]{Remark}

\def\be{\begin{equation}}
\def\ee{\end{equation}}
\def\bea{\begin{eqnarray}}
\def\eea{\end{eqnarray}}

\def\hsm1{\hspace{-1mm}}

\makeatletter
\providecommand\underarrow@[3]{%
  \vtop{\ialign{##\crcr$\m@th\hfil#2#3\hfil$\crcr
  \noalign{\nointerlineskip\kern.12\baselineskip}#1#2\crcr}}}
\providecommand{\underrightarrow}{%
  \mathpalette{\underarrow@\rightarrowfill@}}
\providecommand\rightarrowfill@{\arrowfill@\relbar\relbar\rightarrow}
\providecommand\arrowfill@[4]{%
  $\m@th\thickmuskip0mu\medmuskip\thickmuskip\thinmuskip\thickmuskip
   \relax#4#1\mkern-7mu%
   \cleaders\hbox{$#4\mkern-2mu#2\mkern-2mu$}\hfill
   \mkern-7mu#3$%
}
\makeatother

\begin{document}

\newcommand{\correct}[1]{#1}
\newcommand{\correction}[3]{#3}

\title[Existence of the Lorenz attractor in the Shimizu-Morioka system]{Computer assisted proof of the existence of the Lorenz attractor in the Shimizu-Morioka system}
	
\author{Maciej J. Capi\'nski \footnote[1]{This work was partially supported by the Polish
National Science Centre grants 2015/19/B/ST1/01454, 2016/21/B/ST1/02453 and by the Faculty of Applied Mathematics AGH UST statutory tasks 11.11.420.004 within subsidy of Ministry of Science and Higher Education.}}

\address{Faculty of Applied Mathematics, AGH University of Science and Technology, Mickiewicza 30, 30-059 Krak\'{o}w, Poland}
\ead{maciej.capinski@agh.edu.pl}

\author{Dmitry Turaev \footnote[2]{This work was supported by the RSF grant 14-41-00044. The support of EPSRC is gratefully acknowledged.}}

\address{Imperial College, London SW7 2AZ, United Kingdom, \\
and Lobachevsky University of Nizhni Novgorod, pr. Gagarina 23, 603950, Nizhny Novgorod, Russia}
\ead{d.turaev@imperial.ac.uk}

\author{Piotr~Zgliczy\'nski \footnote[3]{This work was supported by the Polish
National Science Centre grant 2015/19/B/ST1/01454}}

\address{Institute of Computer Science,
Faculty of Mathematics and Computer Science, Jagiellonian University,
 ul. S. {\L}ojasiewicza 6, 30-348 Krak\'ow, Poland}

\ead{umzglicz@cyf-kr.edu.pl, piotr.zgliczynski@ii.uj.edu.pl}

\begin{abstract}
We prove that the Shimizu-Morioka system has a Lorenz attractor for an open set of parameter values. For the proof we employ a criterion proposed by Shilnikov, which allows to conclude the existence of the attractor by examination of the behaviour of only one orbit. The needed properties of the orbit are established by using computer assisted numerics. Our result is also applied to the study of local bifurcations of triply degenerate periodic points of three-dimensional maps. It provides a formal proof of the birth of discrete Lorenz attractors at various global bifurcations.
\end{abstract}


\section{Introduction and main results}
In this paper we provide a solution to a long-standing open problem. We prove, by employing
rigorous numerics, that Shimizu-Morioka system has a Lorenz attractor for an open set of parameter
values. 
\correction{1}{1}{The fact itself is well-known, see \cite{ASh1,ASh2,SST}; however its rigorous proof was missing
which created a formal obstacle to some further developments in the mathematical theory of homoclinic bifurcations (see remarks in Section \ref{dla}). }
Here, we report a computer assisted proof which closes this problem. \correction{2}{2}{}
\correction{3}{3}{Note that the use of computer assistance is sufficiently mild here,
since we employ a Shilnikov criterion that allows us to conclude the existence of the Lorenz attractor in our system by examination of the behavior of only a single orbit of the system (see Sections \ref{shc} and \ref{numres}).}

The Shimizu-Morioka system
\begin{equation}\label{1}
\left\{\begin{array}{l} \dot x=y, \\ \dot y= (1-z) x -\lambda y, \\ \dot z= -\alpha z + x^2,
\end{array}\right.
\end{equation}
where $(x,y,z)$ are coordinates in $\mathbb{R}^3$ and $\alpha>0$, $\lambda>0$ are parameters,
was introduced by T.Shimizu and N.Morioka in \cite{SM} and extensively studied numerically by A. Shilnikov in \cite{ASh1,ASh2}.
One of the main findings was that there is a large open region in the $(\alpha,\lambda)$-plane where this system has a strange attractor very similar to the classical attractor of the Lorenz model (\ref{lor}), see Fig.~\ref{MSfig1}. This is more than just the similarity in shape: as one can infer from the
pictures numerically obtained in \cite{ASh1}, the Poincar\'e map on a two-dimensional cross-section is hyperbolic (very strongly contracting in one direction and expanding in the other
direction), so the attractor in the Shimizu-Morioka system can be described by the Afraimovich-Bykov-Shilnikov geometric Lorenz model \cite{ABS77,ABS82}, i.e., it is a Lorenz attractor.

\begin{figure}[tbp]
\begin{center}
\includegraphics[height=4.4cm]{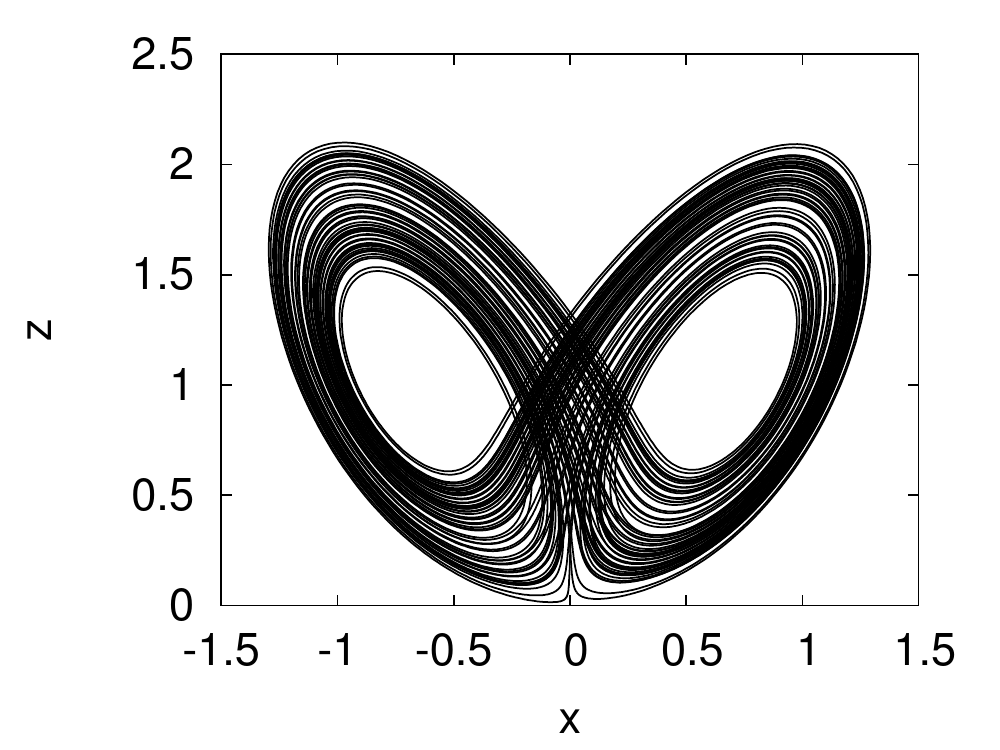}\hspace{0.2cm}\includegraphics[height=4.4cm]{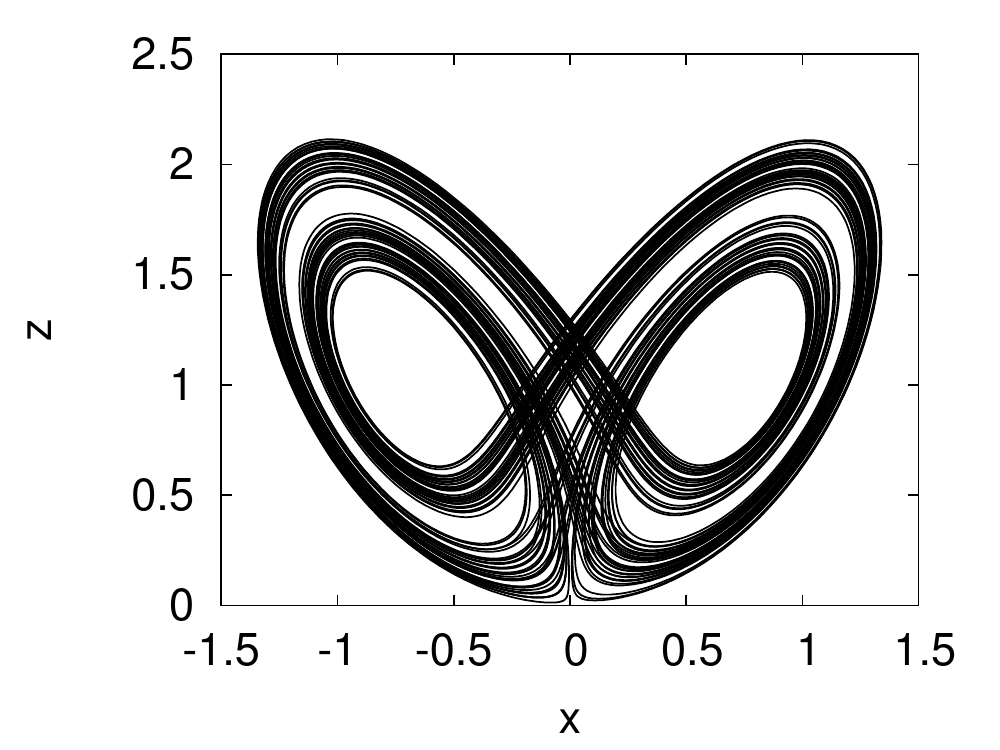}
\end{center}
\caption{Two kinds of Lorenz attractors in the Shimizu-Morioka system: standard (for $\alpha=0.45$, $\lambda=0.9$) on the left, and with a lacuna (for $\alpha=0.5$, $\lambda=0.85$) on the right.}
\label{MSfig1}
\end{figure}

We will describe the geometric Lorenz model and give the corresponding definition of the Lorenz attractors
in a moment. However, we first want to stress that Shimizu-Morioka model is special (and maybe more important than the classical Lorenz model). The reason is that system (\ref{1})
is a truncated normal form for certain codimension-3 bifurcations of equilibria and periodic
orbits \cite{SST,PST}. Because Lorenz or Lorenz-like attractors persist at small perturbations, our result on the existence of the Lorenz attractor in the normal form system (\ref{1}) implies that Lorenz-like attractors
exist for any small perturbation of (\ref{1}). This proves the emergence of
the Lorenz attractor (or its discrete analogue) in the corresponding class of codimension-3 bifurcations.
In particular, in this paper, using Theorem~\ref{thm:main}, we prove the existence of discrete Lorenz attractors in a class of three-dimensional polynomial maps (3D H\'enon maps, see Theorem~\ref{thm:hen}).

\subsection{Pseudohyperbolicity}
In order to talk about Lorenz attractors, we need a proper definition of them. \correction{4}{4}{Classical definitions go back to the Guckenheimer-Williams \cite{G76,GW,W79,Ro} and Afraimovich-Bykov-Shilnikov \cite{ABS77,ABS82} geometric models; modern generalisations can be found in \cite{MPP}. }Here we use
the Afraimovich-Bykov-Shilnikov model for Lorenz attractors, which we describe using the notion of {\em pseudo-hyperbolicity} introduced in \cite{TSH98,TSH07}. A system (a smooth flow or a diffeomorphism) in $\mathbb{R}^n$
is called pseudo-hyperbolic in a strictly forward-invariant domain ${\cal D}\subseteq \mathbb{R}^n$ if:

1) there are directions in which the dynamical system (a flow or a diffeomorphism) is strongly contracting
(``strongly'' means that any possible contraction in transverse directions is always strictly weaker);

2) transverse to the contracting directions the system is volume-expanding (i.e. the volume is stretched
exponentially).

In precise terms, condition 1 reads as follows. For each point of $\cal D$, we assume that there exists a pair
of transversal subspaces $N_1$ and $N_2$ (with $dim(N_2) = k \geq1$ and $dim(N_1) =
n-k$), continuously depending on the point, such that the families of these subspaces
are invariant with respect to the derivative $DX_t$ of the time-$t$ map $X_t$ of the system,
i.e., $D X_t N_1(x) = N_1(X_t(x))$ and $DX_t N_2(x) = N_2(X_t(x))$ for all $t\geq 0$ (in the case of a diffeomorphism, $t$ runs
all positive integer values). We also assume that
there exist constants $C>0$, $\alpha>0$ and $\beta>0$ such that for each $x\in \cal D$ and all $t\geq 0$
\begin{equation}\label{contr}
\|DX_t(x)|_{N_2}\|\leq C e^{-\alpha t}
\end{equation}
and
\begin{equation}\label{doms}
\|DX_t(x)|_{N_2}\| \cdot \|(DX_t(x)|_{N_1})^{-1}\|\leq C e^{-\beta t}.
\end{equation}
The volume-expansion condition 2 reads as follows: there exist constants $C>0$ and $\sigma>0$ such that
for each $x\in \cal D$ and all $t\geq 0$
\begin{equation}\label{expv}
det(DX_t(x)|_{N_1}) \geq C e^{\sigma t}.
\end{equation}

Inequality (\ref{doms}) (the so-called cone condition) ensures that the invariant families of subspaces $N_1$ and
$N_2$ continuously persist for all $C^1$-small perturbations of the system, so the pseudohyperbolic structure is a robust
property of the system. Inequality (\ref{expv}) guarantees that for every orbit in $\cal D$ its maximal Lyapunov exponent
is positive. Therefore, the existence of the pseudohyperbolic structure in a bounded domain $\cal D$ ensures chaotic dynamics in $\cal D$, which cannot be destroyed by small smooth perturbations.

Conditions (\ref{contr}),(\ref{doms}) imply also the existence of a strong-stable invariant foliation ${\cal N}^{ss}$ in $\cal D$,
whose field of tangents is the family $N_2$. For every two points in the same leaf of ${\cal N}^{ss}$ the distance between their
forward orbits tends to zero exponentially, i.e., the forward dynamics of all points in the same leaf are the same.

\subsection{Lorenz attractor}\label{abs}
The Lorenz attractor of the Afraimovich-Bykov-Shilnikov model is the attractor of a pseudohyperbolic system of differential
equations with $dim (N_1)=2$.
Specifically, consider a system $X$ of differential equations with a saddle equilibrium state $O$. Assume that $O$ has a one-dimensional unstable manifold $W^u(O)$ and an $(n-1)$-dimensional stable manifold $W^s(O)$, i.e., if $\lambda_1,\dots,\lambda_n$ are
the eigenvalues of the linearized system at $O$, then $\lambda_1>0$ and $Re \lambda_j <0$ for $j\geq 2$. Assume that
$$\lambda_2 > Re\lambda_j \quad \mbox{   for   } j\geq 3$$
and
$$\lambda_1+\lambda_2 >0.$$
This means that the pseudohyperbolicity conditions (\ref{contr}), (\ref{doms}) and (\ref{expv}) are fulfilled at the point $O$,
with $N_1$ being the two-dimensional eigenspace corresponding to the eigenvalues $\lambda_1$ and $\lambda_2$ and $N_2$ being the eigenspace corresponding to the rest of eigenvalues. We assume that the pseudohyperbolicity property also holds
in a sufficiently large, bounded neighborhood $\cal D$ of $O$. We assume that $\cal D$ is strictly forward-invariant, i.e., there exists $T>0$ such that the image of the closure of $\cal D$ by the time-$T$ map lies strictly inside $\cal D$.

Moreover, we assume  that in $\cal D$ there exists an $(n-1)$-dimensional cross-section $\Pi$ to the flow, such that for every point in $\cal D$  its forward orbit either tends to $O$ (i.e., it lies in $W^s(O)$) or hits $\Pi$ at some moment of time. We assume that
$\Pi$ is divided by a smooth $(n-2)$-dimensional surface $\Pi_0$ into 2 halves, $\Pi_+$ and $\Pi_-$, such that the orbits starting
in $\Pi_+$ and $\Pi_-$ return to $\Pi$ again, while the orbits starting in $\Pi_0$ tend to $O$ as $t\to+\infty$, i.e.,
\correction{5}{5}{$\Pi_0\subset W^s(O)\cap \Pi$}. Thus, the orbits of $X$ define the Poincar\'e map $T:\Pi_+\cup\Pi_-\to \Pi$.

This map is smooth \correction{6}{6}{outside the }discontinuity surface $\Pi_0$. The orbits starting close to $\Pi_0$ pass near $O$, so the return time
tends to infinity as the initial point $M$ tends to $\Pi_0$; note that $TM$ tends to one of the two points, $M_+$ and $M_-$, where $W^u(O)$ intersects $\Pi$. The unstable manifold $W^u(O)$ is one-dimensional, so $W^u(O)\backslash O$ consists of exactly two orbits, $\Gamma_+$ and $\Gamma_-$, called {\em separatrices}, and the points $M_+$ and $M_-$ are the first points where the $\Gamma_+$ and $\Gamma_-$ intersect $\Pi$. We take the convention that $\lim_{M\to\Pi_0} T(M)$ equals to $M_+$ if the initial point $M$ approaches $\Pi_0$ from $\Pi_+$ and $M_-$ if $M$ approaches $\Pi_0$ from $\Pi_-$.

The invariant foliation ${\cal N}^{ss}$ of the system $X$ also corresponds to a codimension-1 strong-stable invariant foliation
${\cal N}^{ss}_\Pi$ for the map $T$: the leaves of ${\cal N}^{ss}_\Pi$ are obtained as intersections with $\Pi$ of the orbits of the leaves of ${\cal N}^{ss}$ by the flow. The foliation is contracting, i.e., the iterations by $T$ of any two points in the same leaf of
${\cal N}^{ss}_\Pi$ converge exponentially to each other. However, the dynamics transverse to the strong-stable foliation are chaotic. The expansion of areas by the flow transverse to ${\cal N}^{ss}$ implies that the Poincar\'e map $T$ is uniformly expanding in the direction transverse to ${\cal N}^{ss}_\Pi$.

In other words, the pseudohyperbolicity of the flow in $\cal D$ implies that the Poincar\'e map $T$ is locally uniformly hyperbolic.
However, one cannot directly apply the theory of uniform hyperbolicity to this map, as it has a singularity (at $\Pi_0$). Nevertheless, the (singular) hyperbolicity of the Poincar\'e map $T$ makes a comprehensive analysis of the structure of the Lorenz attractor in the above described Afraimovich-Bykov-Shilnikov model possible, as it was done in \cite{ABS82}. The results of these papers can be summarized as follows \cite{OT17}:

\begin{theorem}
 \cite{ABS82}  Under conditions above, the set of non-wandering orbits of the system in $\cal D$ consists of a uniquely defined, two-dimensional
closed invariant set ${\cal A}\subset {\cal D}$ (which is called the Lorenz attractor) and a (possibly empty) one-dimensional
closed invariant set $\Sigma$ (which may intersect $\cal A$ but is not a subset of $\cal A$) such that\\
(1) the separatrices $\Gamma_+$ and $\Gamma_-$ and the saddle $O$ lie in $\cal A$;\\
(2) $\cal A$ is transitive, and saddle periodic orbits are dense in $\cal A$;\\
(3) $\cal A$ is the limit of a nested sequence of hyperbolic, transitive, compact invariant sets each of
which is equivalent to a suspension over a finite Markov chain with positive topological
entropy;\\
(4) $\cal A$ is structurally unstable: arbitrarily small smooth perturbations of the system lead to
the creation of homoclinic loops to $O$ and to the subsequent birth and/or disappearance
of saddle periodic orbits within $\cal A$;\\
(5) when $\Sigma=\emptyset$, the set $\cal A$ is the maximal attractor in $\cal D$;\\
(6) when the set $\Sigma$ is non-empty, it is a hyperbolic set equivalent to a suspension over a finite
Markov chain (it may have zero entropy, e.g. be a single saddle periodic orbit);\\
(7) every forward orbit in $\cal D$ tends to ${\cal A} \cup \Sigma$ as $t\to+\infty$;\\
(8) when $\Sigma\neq \emptyset$, the maximal attractor in $\cal D$ is \correction{7}{7}{$cl(W^u(\Sigma)) = {\cal A} \cup W^u(\Sigma)$};\\
(9) $A$ attracts all orbits from its neighbourhood when \correction{8}{8}{${\cal A} \cap \Sigma = \emptyset$}.
\end{theorem}

In \cite{ABS82} the pseudohyperbolicity conditions were expressed in a different (equivalent) form, as explicitly verifiable conditions that ensure the hyperbolicity of the Poincar\'e map $T$. Similar hyperbolicity conditions were checked for the classical Lorenz model
\begin{equation}\label{lor}
\dot x= -10(x-y), \qquad \dot y=x(28-z)-y, \qquad \dot z= -\frac{8}{3} z+ xy
\end{equation}
by W.Tucker, with the use of rigorous numerics. In this way, in \cite{WT1,WT2}, a computer assisted proof of the existence of the Lorenz attractor in system (\ref{lor}) was done.

\subsection{Shilnikov criterion}\label{shc}
In our approach to the proof of the existence of the Lorenz attractor in the \correction{9}{9}{Shimizu-Morioka }system we also rely on the computer assistance, however we need much less computations. We use a criterion proposed by Shilnikov in \cite[p. 240-241]{ShC} that allows to show the existence of the Lorenz attractor by examination of the behavior of only one orbit (a separatrix) of the system of differential equations, instead of the direct check of the hyperbolicity of the Poincar\'e map $T$ which would require a high precision computation of a huge number of orbits.

In \cite[p. 240-241]{ShC}\label{corrected-citation}, several criteria for the birth of the Lorenz attractor were proposed. We use the following one. Consider a system
$X$ of differential \correction{10}{10}{equations }in $\mathbb{R}^n$, which have a saddle equilibrium state $O$ with one-dimensional unstable manifold. Namely, let
$\lambda_1,\dots,\lambda_n$ are the eigenvalues of the linearized system at $O$. We assume that
$$\lambda_1>\; 0\; > \lambda_2 > Re\lambda_j \quad \mbox{   for   } j\geq 3.$$
Let the system be symmetric with respect to a certain involution $\cal R$ and that $O$ is a symmetric equilibrium, i.e.,
${\cal R} O=O$. The eigenvectors $e_1$ and $e_2$ corresponding to the eigenvalues $\lambda_1$ and $\lambda_2$
must be $\cal R$-invariant. We assume that ${\cal R} e_1=- e_1$ and ${\cal R} e_2 = e_2$. This, in particular,
implies that the two unstable separatrices $\Gamma_+$ and $\Gamma_-$, which are tangent at $O$ to $e_1$,
are symmetric to each other, $\Gamma_+ = {\cal R}\Gamma_-$.

Let the system satisfy the following 3 conditions:

1. \correction{11}{11}{Assume that both separatrices $\Gamma_+$ and $\Gamma_-$ return to O as $t\to+\infty$
and are tangent to the vector $e_2$ when entering $O$ (it follows from the symmetry that they form a ``homoclinic butterfly'', i.e., they
enter $O$ from the same direction).}

2. Assume that the so-called {\em saddle value} $\sigma=\lambda_1+\lambda_2$ is zero.

3. Assume that the so-called {\em separatrix value} $A$ satisfies the condition
\begin{equation}\label{maincheck}
0 < |A| < 2.
\end{equation}
The definition of the separatrix value for our case is given in Section \ref{numres}, see (\ref{eqa}).
One can describe $|A|$ as the maximal extent, to which infinitesimal two-dimensional areas can be expanded by the system along the homoclinic loop (the sign of $A$ determines whether the orientation of the areas for which this maximal expansion is achieved
is changed during the propagation along the loop); more about the definition of $A$ can be seen in \cite{Rob1,Rob2, book}.

According to \cite[p. 240-241]{ShC}, bifurcations of systems satisfying the above described conditions lead to the birth of a Lorenz attractor.
To make this statement precise, we note that conditions 1 and 2 describe a codimension-2 bifurcation in the class of
$\cal R$-symmetric systems. Suppose the system is embedded into a two-parameter family of systems $X_{\mu,\varepsilon}$ such that by changing the parameters $\mu$ and $\varepsilon$ we can independently vary the saddle value $\sigma$ near zero and
split the homoclinic loop $\Gamma_+$ (by the symmetry, the homoclinic loop $\Gamma_-$ will be split as well). Then we have

\begin{theorem} \cite[p. 240-241]{ShC}  If condition (\ref{maincheck}) holds at the bifurcation moment (when the system
 $X_{\mu,\varepsilon}$ has a homoclinic butterfly with a zero saddle value), then there exists an open region in the plane of parameters
$\mu,\varepsilon$, for which system $X_{\mu,\varepsilon}$ has a Lorenz attractor of the Afraimovich-Bykov-Shilnikov model.
\end{theorem}

A proof of this result was given by Robinson in \cite{Rob1,Rob2} under certain additional assumptions on the eigenvalues $\lambda$; in full generality this theorem was proven in \cite{OT17}.

We show in Theorems~\ref{th:homoclinic} and \ref{th:sep-val} that Shimizu-Morioka system (\ref{1}) satisfies conditions 1-3 for some value of parameters $(\alpha,\lambda)$ and that the homoclinic loops can be split and the saddle
value can be varied independently when $\alpha$ and $\lambda$ vary. This, as explained, implies our main result:

\begin{theorem}\label{thm:main}
There exists an open set in the plane of parameters $(\alpha,\lambda)$ for which the Shimizu-Morioka system has
a Lorenz attractor.
\end{theorem}

\subsection{Discrete Lorenz attractor in local and global bifurcations.}\label{dla}
The main theorem can be applied to the study of local bifurcations of triply degenerate periodic points of three-dimensional maps.
For example, consider a 3D H\'enon-like map $(x,y,z)\mapsto(\bar x, \bar y, \bar z)$
\begin{equation}\label{3dhe}
\bar x=y, \qquad \bar y=z, \qquad \bar z =M_1 + Bx + M_2y - z^2,
\end{equation}
where $M_1$, $M_2$ and $B$ are parameters and $(x,y,z)\in \mathbb{R}^3$. Pictures numerically obtained in \cite{GOST5}
show that this map has a strange attractor which looks very similar to the Lorenz attractor, even though this is
a discrete dynamical system and not a system of differential equations. The explanation to this fact (the emergence of
a discrete analogue of a Lorenz attractor) can be obtained based on the following observation.

At $M_1=-1/4$, $M_2=1$, $B=1$ this map has a fixed point at $x=y=z=1/2$, and this point has all three multipliers (the eigenvalues of the linearization matrix  at this point) on the unit circle: $(-1,-1,1)$. This is a codimension-3 bifurcation and, as
shown in \cite{GOST5}, the flow normal form for this bifurcation in this map is given by the Shimizu-Morioka system. 
More precisely, at $(M_1,M_2,B)$ close to $(-1/4,1,1)$ and $(B+M_1-1)^2+4M_1>0$ we can shift the coordinate origin
to the fixed point $x=y=z=x^*=(B+M_2-1+\sqrt{(B+M_2-1)^2+4M_1})/2$. Then the map will take the form
\begin{equation}\label{3dh}
\bar x=y, \qquad \bar y=z, \qquad \bar z =Bx + M_2y - 2x^*z- z^2.
\end{equation}
Introduce small parameters $\varepsilon_1=1-B$, $\varepsilon_2=1-M_2$, $\varepsilon_3=2x^*-1$. It was shown in \cite{GOST5} that in the region $\varepsilon_1>0$, $\varepsilon_1+\varepsilon_3>\varepsilon_2$, the second iteration of (\ref{3dh}) near the origin is $O(s^2)$-close, in appropriately chosen rescaled coordinates, to the time-$s$ shift by the flow of a system, which is
$O(s)$-close to the Shimizu-Morioka system (\ref{1}) with $\alpha=(\varepsilon_1+\varepsilon_2+\varepsilon_3)/(4s)$, $\lambda=(\varepsilon_1-\varepsilon_3)/(2s)$,  $s=\sqrt{(\varepsilon_1-\varepsilon_2+\varepsilon_3)/2}$.
This means that if we make $N$ iterations of map (\ref{3dh}), such that $N$ is even and of order $s^{-1}$, the result will be $O(s)$-close to the time-1 map of the Shimizu-Morioka system. 

\correction{12}{12}{In other words, the $N$-th iteration of map (\ref{3dh}) is the time-1 map for a certain small, time-periodic perturbation of the Shimizu-Morioka system:
\begin{equation}\label{1time}
\left\{\begin{array}{l} \dot x=y+ f_1(x,y,z,t), \\ \dot y= (1-z) x -\lambda y+f_2(x,y,z,t), \\ \dot z= -\alpha z + x^2+f_3(x,y,z,t),
\end{array}\right.
\end{equation}
where $f_{1,2,3}$ are $C^1$-small functions, 1-periodic in time $t$. According to \cite{TSH07}, the pseudohyperbolicity persists
at small time-periodic perturbations. Therefore, as by Theorem \ref{thm:main} the Shimizu-Morioka system has
a pseudohyperbolic attractor for an open set of $(\alpha,\lambda)$ values, the same holds true for the map (\ref{3dh}) for the corresponding set of values of $\varepsilon_{1,2,3}$. This attractor is called a discrete Lorenz attractor because it shape is similar
to the Lorenz attractor in the Shimizu-Morioka system. However, its structure is much more complicated than that of the Lorenz attractor described in section \ref{abs}. In particular, the discrete Lorenz attractor may contain homoclinic tangencies \cite{TSH07} and heterodimensional cycles involving saddle periodic orbits with different dimensions of the unstable manifold \cite{LT18}
(more discussions can be found in \cite{GOT13}). This makes a complete description of the dynamics of the discrete Lorenz attractor
impossible but, anyway, its pseudohyperbolicity allows to conclude that every orbit in this attractor is unstable and
this property persists for all small perturbations.}

Since the map (\ref{3dh}) is obtained from (\ref{3dhe}) just by a change of coordinates, we have the following
\begin{theorem}\label{thm:hen}
There exists an open set in the space of parameters $(M_1,M_2,B)$ for which the 3D H\'enon map (\ref{3dhe})
has a pseudohyperbolic discrete Lorenz attractor.
\end{theorem}

Note that the same conclusion holds for larger classes of three-dimensional maps. It was shown in \cite{GOT13} that
the normal form for the bifurcations of the zero fixed point of any map of the type
\begin{eqnarray*}\label{3dhg}
\bar x&= y, \\ \bar y &= z, \\ \bar z &= (1-\varepsilon_1)x + (1-\varepsilon_2) y - (1+\varepsilon_3) z
+ ay^2 + byz + cz^2 + O(\|x,y,z\|^3),
\end{eqnarray*}
is the Shimizu-Morioka system, provided the condition
\begin{equation}\label{scon}
(c-a)(a-b+c)>0
\end{equation}
is fulfilled. Therefore, by the same arguments as for the map (\ref{3dh}), Theorem \ref{thm:main} provides a formal
justification for the claim of \cite{GOT13} (see Lemma 3.1 there) about the existence of pseudohyperbolic attractors
for an open set of parameters $\varepsilon_{1,2,3}$ in maps (\ref{3dhg}) whose coefficients satisfy condition (\ref{scon}).

\correct{Map (\ref{3dh}) is particularly important for the theory of global bifurcations because it appears as a normal form
 for the first-return maps near many types of homoclinic tangencies and heteroclinic cycles with tangencies in three-dimensional and higher-dimensional maps \cite{GGS}. Given a homoclinic or heteroclinic cycle, one can estimate its effective dimension --- the maximal possible number of zero Lyapunov exponents that periodic orbits born at bifurcations of such cycle can have \cite{T96}. Several different classes of maps with effectively 3-dimensional homoclinic or heteroclinic cycles were considered in \cite{GMO,GST9,GI,GIT,GI17,O17}. It was shown in these papers that the corresponding bifurcations should produce pseudohyperbolic (discrete Lorenz) attractors. Namely, it was shown that a generic three-parameter unfolding of each of these bifurcations creates regions in the phase space where the first-return map is closely (as close as one wants) approximated, in appropriately chosen coordinates, by the 3D-H\'enon map (\ref{3dh}). Therefore, by the robustness of the pseudohyperbolicity property, the birth of the pseudohyperbolic Lorenz-like attractors at each of these bifurcations is established once the existence of a pseudohyperbolic attractor is shown in map (\ref{3dh}).

As we mentioned, the first numerical evidence for the existence of such attractor in map (\ref{3dh}) was obtained in \cite{GOST5}.
Theorem \ref{thm:hen} makes it a rigorous result. Thus, now we have a full formal proof to the results of \cite{GMO,GST9,GI,GIT,GI17,O17} about the emergence of pseudohyperbolic attractors at bifurcations of effectively 3-dimensional homoclinic and heteroclinic cycles, including the birth of infinitely many coexisting pseudohyperbolic attractors in some situations, see \cite{GST9,GI}.}

\subsection{Main results of the rigorous numerics}\label{numres}
We establish Theorem \ref{thm:main} in the following steps. First, we find  good bounds for the values of
parameters $(\alpha,\lambda)$ for which Shimizu-Morioka system (\ref{1}) \correction{13}{13}{satisfies conditions }1 and 2 of Section \ref{shc}
(the existence of a homoclinic butterfly with zero saddle value $\sigma$). It is easy to check that condition $\sigma=0$
reads as
\begin{equation}\label{sv}
\lambda=\frac{1}{\alpha}-\alpha.
\end{equation}
We will study how the separatrices $\Gamma_\pm$ move as $\alpha$ varies while $\lambda$ is given by (\ref{sv});
the moment a separatrix forms a homoclinic loop corresponds to the parameter values we are looking for.

\correction{14}{14}{It is convenient to scale time $t\to t/\alpha =T$ and variables $x\to x/\sqrt{\alpha}=X$, $y\to y/(\alpha\sqrt{\alpha})=Y$}. Then the system
takes the form
\begin{eqnarray}
\dot{X}& =Y,  \nonumber \\
\dot{Y}& =\left( a+1\right) \left( 1-Z\right) X-aY,  \label{eq:S-M-ode} \\
\dot{Z}& =-Z+X^{2}, \nonumber
\end{eqnarray}
where
$$a=\frac{1}{\alpha^2}-1=\lambda/\alpha.$$

The eigenvalues of the linearization matrix at $(0,0,0)$ are $(-1,1,-(1+a))$ with corresponding
eigenvectors given by $(0,0,1)$, $(1,1,0)$ and $(-(1+a)^{-1},1,0)$,
respectively. We will investigate (\ref{eq:S-M-ode}) for $a\approx 1.72$,
so the hyperbolic equilibrium state $(0,0,0)$ has a one dimensional
unstable manifold, and a two dimensional stable manifold. When the manifolds
intersect, we have a solution that tends to zero both as $t\to+\infty$ and $t\to-\infty$, i.e., a
homoclinic loop to the equilibrium state.
\begin{theorem}
\label{th:homoclinic}There exists $a=a_{0}\in \left[ a_{l},a_{r}\right]$, where
\begin{eqnarray*}
a_{l}& =1.72432329151541-10^{-13}, \\
a_{r}& =1.72432329151541+10^{-13},
\end{eqnarray*}%
for which (\ref{eq:S-M-ode}) has a homoclinic orbit $X_{0}\left( t\right)
,Y_{0}\left( t\right) ,Z_{0}\left( t\right) $ to the equilibrium state $(0,0,0)$.
As $a$ changes between $a_l$ and $a_r$, the loop splits.
Taking
\begin{equation*}
\xi =1-10^{-4},\qquad c=3.5,\qquad T=26,
\end{equation*}
we have the bound:
\begin{equation*}
\left\vert X_{0}\left( T+t\right) \right\vert ,\left\vert Y_{0}\left(
T+t\right) \right\vert ,\left\vert Z_{0}\left( T+t\right) \right\vert \leq
ce^{-\xi t}\left\Vert \left( X_{0}\left( T\right) ,Y_{0}\left( T\right)
,Z_{0}\left( T\right) \right) \right\Vert ,
\end{equation*}
for $t\geq 0$, and for $t\leq 0$ we have:
\begin{equation*}
\left\vert X_{0}\left( t\right) \right\vert ,\left\vert Y_{0}\left( t\right)
\right\vert ,\left\vert Z_{0}\left( t\right) \right\vert \leq ce^{-\xi
\left\vert t\right\vert }\left\Vert \left( X_{0}\left( 0\right) ,Y_{0}\left(
0\right) ,Z_{0}\left( 0\right) \right) \right\Vert .
\end{equation*}
\end{theorem}

The proof of the above theorem is obtained with a computer assistance and is given in section \ref{sec:CAPhomoclinic}.
The mere existence of a homoclinic loop in system (\ref{eq:S-M-ode}) can be obtained  purely analytically \cite{TT}.
However, we need good estimates on $a_0$ and the corresponding homoclinic solution, which the methods of \cite{TT}
do not provide. Crucially, these estimates are used in the next theorem where the separatrix value $A$ is estimated.

The separatrix value can be defined as follows (cf. \cite{Rob1,Rob2,TT}). Let a system of three differential equations
have a homoclinic solution $(X_0(t), Y_0(t), Z_0(t))$ to a hyperbolic equilibrium state at zero, so
$(X_0(t), Y_0(t), Z_0(t))\to 0$ as  $t\to\pm\infty$. Let
\begin{equation}\label{eq12}
\frac{d}{dt}\left(\begin{array}{c}x \\ y \\ z\end{array}\right)= B(t) \left(\begin{array}{c}x \\ y \\ z\end{array}\right)
\end{equation}
be the linearization of the system along the loop. In particular, in the case of system  (\ref{eq:S-M-ode}) we have
$$B =\left(\begin{array}{ccc} 0 & 1 & 0\\
(a_0 + 1)(1 - Z_0(t))  & -a_0 & -(a_0 + 1)X_0(t)\\
2X_0(t) & 0 &-1\end{array}\right).$$
Let $\xi_1$, $\xi_2$ be any two vectors and let $\eta = \xi_1 \times \xi_2$ be their vector
product. If the evolution of $\xi_1$ and $\xi_2$ is defined by (\ref{eq12}), then the
evolution of $\eta$ is governed by
\begin{equation}\label{areas}
\frac{d\eta}{dt} = -(B^\top - tr(B) I) \eta,
\end{equation}
where $I$ is the $(3 \times 3)$ identity matrix. This equation describes the evolution of infinitesimal two-dimensional areas near the homoclinic loop.

Since $(X_0(t), Y_0(t), Z_0(t))$ tends to zero exponentially, the asymptotic behavior of solutions of (\ref{areas}) as $t\to\pm\infty$
is determined by the limit matrix
$$\hat B = -(B_\infty^\top - tr(B_\infty) I)$$
where $B_\infty=\lim_{t\to\pm\infty} B(t)$, which is the linearization of the original system at the hyperbolic equilibrium at zero. If
$\lambda_1 >0 > \lambda_2 >\lambda_3$ are the eigenvalues of $B_\infty$, then the eigenvalues of $\hat B$ are
$$\sigma_1=\lambda_1+\lambda_2, \qquad \sigma_2=\lambda_1+\lambda_3, \qquad \sigma_3=\lambda_2+\lambda_3.$$
We are interested here in the case of zero saddle value, i.e., $\lambda_1+\lambda_2=0$. Then the eigenvalues of $\hat B$ are
$0$, $\sigma_2<0$ and $\sigma_3 < 0$, so every solution of (\ref{areas}) tends, as $t\to+\infty$ to a constant times the eigenvector of $\hat B$ that corresponds to the zero eigenvalue (this is the vector $v_0$, which is the vector product
of the eigenvectors of $B_\infty$ that correspond to the eigenvalues $\lambda_1$ and $\lambda_2$). It also follows that
only one solution of (\ref{areas}) tends to $v_0$ in backward time, as $t\to -\infty$. We take this particular solution $\eta_0(t)$
and denote
\begin{equation}\label{eqa}
\lim_{t\to+\infty} \eta_0(t) = A v_0.
\end{equation}
The coefficient $A$ is the sought separatrix value. \correction{15}{15}{From the definition of $A$ one can see that
$$|A| = \sup \lim_{t\to+\infty} \frac{\|\eta (t)\|}{\|\eta(-t)\|},$$
where the supremum is taken over all the solutions of (\ref{areas}). }Thus, $A$ determines
the maximal expansion of infinitesimal areas along the homoclinic loop.

In the case of system  (\ref{eq:S-M-ode}) system (\ref{areas}) becomes
\begin{equation}
\eta ^{\prime }\left( t\right) =\left(
\begin{array}{ccc}
-\left( a_{0}+1\right) & -\left( a_{0}+1\right) \left( 1-Z_{0}\left(
t\right) \right) & -2X_{0}\left( t\right) \\
-1 & -1 & 0 \\
0 & \left( a_{0}+1\right) X_{0}\left( t\right) & -a_{0}%
\end{array}%
\right) \eta \left( t\right).  \label{eq:S-M-ode2}
\end{equation}
This equation, in the limit $t\rightarrow \pm \infty$ becomes
\begin{equation}
\eta ^{\prime }\left( t\right) =\hat B \eta \left( t\right) = \left(
\begin{array}{ccc}
-(a_{0}+1) & -(a_{0}+1) & 0 \\
-1 & -1 & 0 \\
0 & 0 & -a_{0}%
\end{array}%
\right) \eta \left( t\right) .  \label{eq:S-M-ode2-limit}
\end{equation}%
The eigenvalues are $0,-(a+2),-a$ with corresponding eigenvectors $%
v_{0}=(1,-1,0)$, $v_{-\left( a+2\right) }=(1+a,1,0)$ and $v_{-a}(0,0,1)$,
respectively.

\begin{theorem}
\label{th:sep-val}\correction{16}{16}{There exists an orbit $\eta \left( t\right) $ of (\ref{eq:S-M-ode2}), for which
\begin{eqnarray*}
\lim_{t\rightarrow -\infty }\eta \left( t\right) &\in &\left[
0.99984336210766,1.0001566378923\right] v_{0}, \\
\lim_{t\rightarrow +\infty }\eta \left( t\right) &\in
&[0.62606812264791,0.62663392848044] v_{0},
\end{eqnarray*}
meaning that
\begin{equation}\label{etaa}
\lim_{t\rightarrow +\infty }\frac{\left\Vert \eta \left( t\right)
\right\Vert }{\left\Vert \eta \left( -t\right) \right\Vert }\in \left[
0.62597007201516,0.6267320984754\right].
\end{equation}}
\end{theorem}

The proof of the above theorem is obtained with computer assistance and is given in Section \ref{sec:CAP-separatrix-value}.
By definition (\ref{eqa}) of the separatrix value, estimate (\ref{etaa}) gives the following bounds for the separatrix value of the
homoclinic loop in system (\ref{eq:S-M-ode}):
$$A\in  \left[0.62597007201516,0.6267320984754\right].$$
Importantly $0<A<1$. So, by applying Shilnikov criterion we obtain our main result, Theorem \ref{thm:main}, see section \ref{shc}.

The plot of the homoclinic trajectory from Theorem \ref{th:homoclinic} can
be seen on the left plot in Figure \ref{fig:original-coordinates} on page
\pageref{fig:original-coordinates}. The plot of the heteroclinic orbit from
Theorem \ref{th:sep-val} is given in Figure \ref{fig:eta-trajectory} on page
\pageref{fig:eta-trajectory}.

\begin{remark}\label{rem:waves}
\correction{18}{18}{The techniques we use for the proofs of Theorems \ref{th:homoclinic} and \ref{th:sep-val} could also be used for the study of traveling waves and their stability in problems coming from PDEs on the line \cite{wave1,wave2,wave3,wave4}. This is because proving their existence requires establishing the existence of heteroclinic/homoclinic connecting orbits. Investigating their stability requires studying additionally some linearized equations. We develop tools for such problems in Sections \ref{sec:homoclinic}, \ref{sec:sep-val}.}
\end{remark}


\section{Some notations}

In the subsequent sections we present a methodology for establishing homoclinic orbits and for the computation of the separatrix value.  First we introduce the following notations. 

We will write $B_{k}(R)$
for ball in $\mathbb{R}^{k}$ of radius $R$, centered at zero. For a matrix $A
$ we define the logarithmic norm of $A$ as (see \cite{CZmelnikov} and literature cited there)
\begin{equation*}
l\left( A\right) =\lim_{h\rightarrow 0^{+}}\frac{\left\Vert I+Ah\right\Vert
-\left\Vert I\right\Vert }{h}.
\end{equation*}%
We will also use the notation%
\begin{eqnarray*}
m\left( A\right)  &=&\left\{
\begin{array}{l}
\frac{1}{\left\Vert A\right\Vert ^{-1}}\qquad \det A\neq 0, \\
0\qquad \qquad \text{otherwise,}%
\end{array}%
\right.  \\
m_{l}\left( A\right)  &=&-l\left( -A\right) .
\end{eqnarray*}%
The $m(A)$ is a number with the property that $\left\Vert Av\right\Vert \geq
m\left( A\right) \left\Vert v\right\Vert .$ The $m_{l}\left( A\right) $ can
be interpreted as a `bound from below' of the logarithmic norm.

Let $u,s\in \mathbb{N}$. For a function $f:\mathbb{R}^{u}\times \mathbb{R}%
^{s}\rightarrow \mathbb{R}^{u}\times \mathbb{R}^{s}$, we will use the
notations $\left( x,y\right) \in \mathbb{R}^{u}\times \mathbb{R}^{s}$, where
$x\in \mathbb{R}^{u}$ and $y\in \mathbb{R}^{s}$. The $x$ will play the role
of an ``unstable" coordinate and $y$ will be the ``stable" coordinate, hence
the choice of the notation $u,s$.  We will also write $f_{x}$, and $f_{y}$
for the projections of $f$ onto $\mathbb{R}^{u}$ and $\mathbb{R}^{s}$,
respectively. For a set $D\subset \mathbb{R}^{u}\times \mathbb{R}^{s}$ we
define%
\begin{eqnarray*}
\left[ \frac{\partial f_{x}}{\partial x}\left( D\right) \right] &:&=\left\{
A=\left( a_{ij}\right) \in \mathbb{R}^{u\times u}:a_{i,j}\in \left[
\inf_{p\in D}\frac{\partial f_{x_{i}}}{\partial x_{j}}\left( p\right)
,\sup_{p\in D}\frac{\partial f_{x_{i}}}{\partial x_{j}}\left( p\right) %
\right] \right\} , \\
\left[ \frac{\partial f_{x}}{\partial y}\left( D\right) \right] &:&=\left\{
A=\left( a_{ij}\right) \in \mathbb{R}^{u\times s}:a_{i,j}\in \left[
\inf_{p\in D}\frac{\partial f_{x_{i}}}{\partial y_{j}}\left( p\right)
,\sup_{p\in D}\frac{\partial f_{x_{i}}}{\partial y_{j}}\left( p\right) %
\right] \right\} .
\end{eqnarray*}%
We also define%
\begin{eqnarray*}
m\left( \frac{\partial f_{x}}{\partial x}\left( D\right) \right)
&:=&\inf_{A\in \left[ \frac{\partial f_{x}}{\partial x}\left( D\right) %
\right] }m\left( A\right) , \\
m_{l}\left( \frac{\partial f_{x}}{\partial x}\left( D\right) \right)
&:=&\inf_{A\in \left[ \frac{\partial f_{x}}{\partial x}\left( D\right) %
\right] }m_{l}\left( A\right) , \\
\left\Vert \frac{\partial f_{x}}{\partial y}\left( D\right) \right\Vert
&:=&\sup_{A\in \left[ \frac{\partial f_{x}}{\partial y}\left( D\right) %
\right] }\left\Vert A\right\Vert .
\end{eqnarray*}

We will use the notation $\mathrm{int}(D)$, $\overline{D}$ and $\partial D$
for the interior, closure and boundary of a set $D$, respectively.


\section{Establishing homoclinics\label{sec:homoclinic}}

In this section we give an overview of the method for establishing the
existence of homoclinic orbits to fixed points. The method is written for
the case where we consider a parameter dependent ODE with the vector field $%
f:\mathbb{R}^{3}\times \mathbb{R}\rightarrow \mathbb{R}^{3}$, 
\begin{equation}
p^{\prime }=f(p,a).  \label{eq:ode_Shil}
\end{equation}%
and $a\in A$ is a parameter, with $A=[a_{l},a_{r}]\subset \mathbb{R}$. We
assume that for each $a\in A$ (\ref{eq:ode_Shil}) has a hyperbolic fixed
point $p_{a}^{\ast }$, with one dimensional unstable manifold and two
dimensional stable manifolds.
(If dimensions are the other way around we can change the sign of the vector
field.)

Let $\Phi _{t}(p,a)$ be the flow induced by (\ref{eq:ode_Shil}). Let $%
\overline{B}_{u}\left( R\right) =\left[ -R,R\right] \subset \mathbb{R}$, $%
\overline{B}_{s}\left( R\right) \subset \mathbb{R}^{2}$ and let%
\begin{equation*}
D=\overline{B}_{u}\left( R\right) \times \overline{B}_{s}\left( R\right)
\subset \mathbb{R}^{3},
\end{equation*}%
be a neighborhood of the smooth family of fixed points, meaning that we
assume $p_{a}^{\ast }\in \mathrm{int}D$ for any $a\in A$. 
\begin{figure}[tbp]
\begin{center}
\includegraphics[height=3cm]{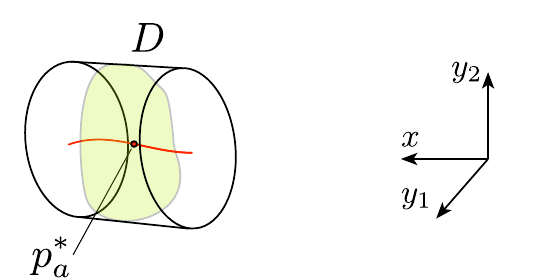}
\end{center}
\caption{The local unstable manifold $W_{a }^{u}$ in red, and
the local stable manifold $W_{a }^{s}$ in green.}
\label{fig:D-setup}
\end{figure}

We denote by $W_{a}^{u}$ the local unstable manifold of $p_{a}^{\ast }$ in $%
D $ and by $W_{a}^{s}$ the local stable manifold of $p_{a}^{\ast }$ in $D$,
i.e.%
\begin{eqnarray}
W_{a}^{u}& =\left\{ p\in D:\Phi _{t}\left( p,a\right) \in D\text{ for }t\leq
0\text{ and }\lim_{t\rightarrow -\infty }\Phi _{t}\left( p,a\right)
=p_{a}^{\ast }\right\} ,  \label{eq:Wu-def} \\
W_{a}^{s}& =\left\{ p\in D:\Phi _{t}\left( p,a\right) \in D\text{ for }t\geq
0\text{ and }\lim_{t\rightarrow +\infty }\Phi _{t}\left( p,a\right)
=p_{a}^{\ast }\right\} .  \label{eq:Ws-def}
\end{eqnarray}%
We assume that $W_{a}^{u}$ and $W_{a}^{s}$ are graphs of $C^{1}$ functions%
\begin{eqnarray*}
w_{a}^{u}& :\overline{B}_{u}\left( R\right) \rightarrow \overline{B}%
_{s}\left( R\right) , \\
w_{a}^{s}& :\overline{B}_{s}\left( R\right) \rightarrow \overline{B}%
_{u}\left( R\right) ,
\end{eqnarray*}%
meaning that (see Figure \ref{fig:D-setup})%
\begin{eqnarray}
W_{a}^{u}& =\left\{ \left( x,w_{a}^{u}\left( x\right) \right) :x\in 
\overline{B}_{u}\left( R\right) \right\} ,  \nonumber \\
W_{a}^{s}& =\left\{ \left( w_{a}^{s}\left( y\right) ,y\right) :y\in 
\overline{B}_{s}\left( R\right) \right\} .  \label{eq:Ws-graph}
\end{eqnarray}

\begin{figure}[ptb]
\begin{center}
\includegraphics[height=4cm]{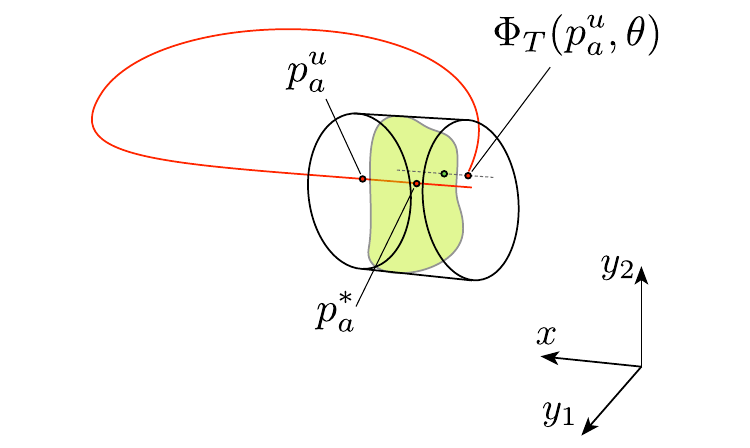}
\end{center}
\caption{We have the $1$-dimensional unstable manifold of $p_{a}^{\ast }$ in red, and the $2$-dimensional local stable manifold $W_{%
a}^{s}$ in $D$ in green. The $h\left( a\right) $
is the signed distance along the $x$ coordinate between $W_{a}^{s}$ and $\Phi_{T}\left( p_{a}^{u},a\right) $;
this is the distance along the dotted line on the plot. }
\label{fig:homoclinic-2}
\end{figure}

Let 
\begin{equation}
p_{a}^{u}:=\left( R,w_{a}^{u}\left( R\right) \right) \in \mathbb{R}^{3}.
\label{eq:pu-theta-def}
\end{equation}\correction{17}{17}{}
Consider $T>0$ and assume that for all $a\in A$, $\Phi _{T}\left(
p_{a}^{u},a\right) \in D$. Let us define%
\begin{equation*}
h:A\rightarrow \mathbb{R},
\end{equation*}%
as%
\begin{equation}
h\left( a\right) =\pi _{x}\Phi _{T}\left( p_{a}^{u},a\right) -w_{a}^{s}(\pi
_{y}\Phi _{T}\left( p_{a}^{u},a\right) ).  \label{eq:h-function-def}
\end{equation}

We now state a natural result, that $h\left( a\right) =0$ implies an
intersection of the stable and unstable manifolds of $p_{a}^{\ast }$. (See
Figure \ref{fig:homoclinic-2}.)

\begin{theorem}
\cite{CW}\label{th:homoclinic-existence}If 
\begin{equation}
h(a_{l})<0\qquad \text{and}\qquad h(a_{r})>0  \label{eq:Bolzano-assmpt}
\end{equation}%
then there exists a $\psi \in A$ for which we have a homoclinic orbit to $%
p_{\psi }^{\ast }$.
\end{theorem}

\subsection{Computer assisted bounds for unstable manifolds of fixed points
of ODEs\label{sec:Wu-bounds}}

In order to apply Theorem \ref{th:homoclinic-existence}, we need to be able
to establish bounds for $h$ defined in (\ref{eq:h-function-def}). Using a
rigorous, interval arithmetic based integrator\footnote{%
in our application we use the CAPD package: http://capd.ii.uj.edu.pl/}, it
is possible to obtain rigorous enclosures for $\Phi _{T}$. In this section
we discuss how to obtain bounds on parameterizations $w_{a}^{s},w_{a}^{u}$
of stable and unstable manifolds of fixed points. For simplicity, we will
skip the parameter $a$ and consider an ODE%
\begin{equation}
p^{\prime }=f(p).  \label{eq:ode-wu-ws}
\end{equation}%
When combined with Theorem \ref{th:homoclinic-existence}, we can apply below
results for (\ref{eq:ode_Shil}) with fixed $a.$

Let $D\subset \mathbb{R}^{u}\times \mathbb{R}^{s}$, 
\begin{equation*}
D=\overline{B}_{u}\left( R\right) \times \overline{B}_{s}\left( R\right) .
\end{equation*}%
In this section we do not need to assume that $u=1$ and $d=2.$ The result
works for arbitrary dimensions.

We define%
\begin{eqnarray}
\overrightarrow{\mu } &=&\sup_{z\in D}\left\{ l\left( \frac{\partial f_{y}}{%
\partial y}(z)\right) +\frac{1}{L}\left\Vert \frac{\partial f_{y}}{\partial x%
}(z)\right\Vert \right\} ,  \label{eq:mu-arrow-def} \\
\overrightarrow{\xi } &=&m_{l}\left( \frac{\partial f_{x}}{\partial x}%
(D)\right) -L\left\Vert \frac{\partial f_{x}}{\partial y}(D)\right\Vert .
\label{eq:xi-arrow-def}
\end{eqnarray}

\begin{definition}
We say that the vector field $f$ satisfies rate conditions if%
\begin{equation}
\overrightarrow{\mu }<0<\overrightarrow{\xi },  \label{eq:rate-cond}
\end{equation}
\end{definition}

\begin{definition}
We say that $D=\overline{B}_{u}\left( R\right) \times \overline{B}_{s}\left(
R\right) $ is an \correction{19}{19}{isolating block for }(\ref{eq:ode-wu-ws}) if

\begin{enumerate}
\item For any $q\in\partial\overline{B}_{u}\left( R\right) \times \overline{B%
}_{s}\left( R\right) $, 
\begin{equation*}
\left( \pi_{x}f(q)|\pi_{x}q\right) >0.
\end{equation*}

\item For any $q\in\overline{B}_{u}\left( R\right) \times\partial \overline{B%
}_{s}\left( R\right) $,%
\begin{equation*}
\left( \pi_{y}f(q)|\pi_{y}q\right) <0.
\end{equation*}
\end{enumerate}
\end{definition}

\begin{definition}
We define the unstable set in $D$ as%
\begin{equation*}
W^{u}=\{z:\text{ }\Phi _{t}(z)\in D\text{ for all }t<0\}.
\end{equation*}
\end{definition}

\correction{20}{20}{Below theorem is a simplified (adapted to fixed points) version of the
results from \cite[Theorem 30]{CZmelnikov}. The paper \cite{CZmelnikov} is
in the setting of normally hyperbolic invariant manifolds, hence the Theorem
30 from that paper is more involved than below result.
\begin{theorem}\cite[Theorem 30]{CZmelnikov}
\label{th:Wu-cap-bound}Let $k\geq 1$. Assume that $f$ is $C^{1}$ and
satisfies the rate conditions. Assume also that $D=\overline{B}_{u}\left(
R\right) \times \overline{B}_{s}\left( R\right) $ is an isolating block
for $f$. Then the set $W^{u}$ is a manifold, which is a graph over $%
\overline{B}_{u}\left( R\right) $. To be more precise, there exists a $C^{1}$
function 
\begin{equation*}
w^{u}:\overline{B}_{u}(R)\rightarrow \overline{B}_{s}(R),
\end{equation*}%
such that 
\begin{equation}
W^{u}=\left\{ \left( x,w^{u}(x)\right) :x\in \overline{B}_{u}(R)\right\} .
\label{eq:Wu-as-graph}
\end{equation}%
Moreover, $w^{u}$ is Lipschitz with constant $L$.
\end{theorem}}

\begin{proof}
The proof is given in Appendix \ref{sec:Wu-cap-bound-proof}. It is a
modification of the argument from \cite{CZmelnikov}. The main difference is
that the results from \cite{CZmelnikov} are for the setting in which in
addition to the hyperbolic coordinates $x,y$ we have a centre coordinate.
Due to this the Lipschitz bound from Theorem \ref{th:Wu-cap-bound} is
sharper compared with \cite{CZmelnikov}. In the proof we focus on this issue.
\end{proof}

We will now discuss the contraction rate along the stable manifold $W^{u}$
from Theorem \ref{th:Wu-cap-bound}. First we shall need an auxiliary result.
Consider a flow $\Phi _{t}\left( z\right) $, for $\Phi :\mathbb{R\times R}%
^{u}\times \mathbb{R}^{s}\rightarrow \mathbb{R}^{u}\times \mathbb{R}^{s}$, 
\begin{equation*}
\Phi _{t}\left( x,y\right) =\left( \pi _{x}\Phi _{t}\left( x,y\right) ,\pi
_{y}\Phi _{t}\left( x,y\right) \right) ,
\end{equation*}%
and define the following constants \correction{21}{21}{
\begin{eqnarray}
\mu \left( h\right) &=&\sup_{z\in D}\left\{ \left\Vert \frac{\partial \pi
_{y}\Phi }{\partial y}\left( h,z\right) \right\Vert +\frac{1}{L}%
\left\Vert \frac{\partial \pi _{y}\Phi }{\partial x}\left( h,z\right)
\right\Vert \right\} ,  \label{eq:mu-h} \\
\xi \left( h\right) &=&m\left[ \frac{\partial \pi _{x}\Phi }{\partial x}%
(h,D)\right] -L\sup_{z\in D}\left\Vert \frac{\partial \pi _{x}\Phi }{%
\partial y}\left( h,z\right) \right\Vert .  \label{eq:xi-h}
\end{eqnarray}}

\begin{theorem}
\label{th:coeff-ode-map}\cite[Theorem 31]{CZmelnikov}Let $\overrightarrow{%
\xi }$ and $\overrightarrow{\mu }$ be the constants defined in (\ref%
{eq:mu-arrow-def}) and (\ref{eq:xi-arrow-def}). If $\Phi _{t}$ is the flow
induced by (\ref{eq:ode-wu-ws}), then for $h>0$%
\begin{eqnarray*}
\mu \left( h\right) &=&1+h\overrightarrow{\mu }+O\left( h^{2}\right) , \\
\xi \left( h\right) &=&1+h\overrightarrow{\xi }+O\left( h^{2}\right) .
\end{eqnarray*}
\end{theorem}

\begin{theorem}
\label{th:contraction-bound}Let $W^{u}$ be the manifold established in
Theorem \ref{th:Wu-cap-bound}. Then for any $p_{1},p_{2}\in W^{u}$%
\begin{equation*}
\left\Vert \Phi _{-t}\left( p_{1}\right) -\Phi _{-t}\left( p_{2}\right)
\right\Vert \leq ce^{-\overrightarrow{\xi }t}\left\Vert \pi _{x}\left(
p_{1}-p_{2}\right) \right\Vert \qquad \text{for all }t\geq 0,
\end{equation*}%
for $c=2\sqrt{1+L^{2}}.$
\end{theorem}

\begin{proof}
Let $p_{1},p_{2}\in W^{u}$. Since $w^{u}$ is Lipschitz with constant $L$, 
\begin{equation}
\left\Vert \pi _{y}\left[ p_{1}-p_{2}\right] \right\Vert \leq L\left\Vert
\pi _{x}\left[ p_{1}-p_{2}\right] \right\Vert .  \label{eq:Lip-ineq-tmp-1}
\end{equation}

Let $q_{1},q_{2}\in W^{u}$. If $\Phi _{t}(q_{1}),\Phi _{t}(q_{2})\in D$ for $%
t\in (0,T]$, then for $0<h\leq T$ holds 
\begin{eqnarray*}
\Vert \pi _{x}(\Phi _{h}(q_{1})-\Phi _{h}(q_{2}))\Vert &\geq &\xi (h)\Vert
\pi _{x}(q_{1}-q_{2})\Vert \\
&=&(1+h\overrightarrow{\xi }+O(h^{2}))\Vert \pi _{x}(q_{1}-q_{2})\Vert .
\end{eqnarray*}%
If we take $h=\frac{T}{N}$, then 
\begin{eqnarray*}
\Vert \pi _{x}(\Phi _{T}(q_{1})-\Phi _{T}(q_{2}))\Vert &\geq &(1+h%
\overrightarrow{\xi }+O(h^{2}))^{N}\Vert \pi _{x}(q_{1}-q_{2})\Vert \\
&\rightarrow &e^{T\overrightarrow{\xi }}\Vert \pi _{x}(q_{1}-q_{2})\Vert
,\quad N\rightarrow \infty .
\end{eqnarray*}%
Observe that from the above it follows that (we set $p_{i}=\Phi _{-T}(q_{i})$%
) 
\begin{equation}
\Vert \pi _{x}(p_{1}-p_{2})\Vert \geq e^{T\overrightarrow{\xi }}\Vert \pi
_{x}(\Phi _{-T}(p_{1})-\Phi _{-T}p_{2})\Vert
\label{eq:contraction-bound-tmp-1}
\end{equation}%
for any $T>0$.

Using (\ref{eq:Lip-ineq-tmp-1}) in the third line and (\ref%
{eq:contraction-bound-tmp-1}) in the last line,%
\begin{eqnarray*}
&&\left\Vert \Phi _{-T}\left( p_{1}\right) -\Phi _{-T}\left( p_{2}\right)
\right\Vert ^{2} \\
&=&\left\Vert \pi _{x}\left[ \Phi _{-T}\left( p_{1}\right) -\Phi _{-T}\left(
p_{2}\right) \right] \right\Vert ^{2}+\left\Vert \pi _{y}\left[ \Phi
_{-T}\left( p_{1}\right) -\Phi _{-T}\left( p_{2}\right) \right] \right\Vert
^{2} \\
&\leq &\left( 1+L^{2}\right) \left\Vert \pi _{x}\left[ \Phi _{-T}\left(
p_{1}\right) -\Phi _{-T}\left( p_{2}\right) \right] \right\Vert ^{2} \\
&\leq &\left( 1+L^{2}\right) e^{-2T\overrightarrow{\xi }}\left\Vert \pi _{x}%
\left[ p_{1}-p_{2}\right] \right\Vert ^{2},
\end{eqnarray*}%
which concludes the proof.
\end{proof}

\begin{remark}
Theorems \ref{th:Wu-cap-bound}, \ref{th:contraction-bound} can also be
applied to establish bounds on the stable manifold. In order to do so, it is
enough to consider $p^{\prime }=-f(p)$ instead of (\ref{eq:ode-wu-ws}), and
to swap the roles of the coordinates $x,y$. The unstable manifold for the
vector field $-f$ is the stable manifold for $f$.
\end{remark}

\subsection{Computer assisted proof of homoclinic intersection\label%
{sec:CAPhomoclinic}}

In this section we give an overview of the computer assisted proof of
Theorem \ref{th:homoclinic}.

To apply the method from Sections \ref{sec:homoclinic}, \ref{sec:Wu-bounds}
to conduct a computer assisted proof we follow the steps, as outlined in 
\cite{CW}: \medskip

\noindent \textbf{Algorithm 1.}

\begin{enumerate}
\item \label{step:Wu} In local coordinates around zero, using Theorem \ref%
{th:Wu-cap-bound}, establish the bounds on the unstable manifolds for the
family of vector fields .

\item By changing sign of the vector field, using the same procedure as in
step \ref{step:Wu}, establish bounds on the stable manifolds.

\item Propagate the bounds on the unstable manifold along the flow, and
establish the homoclinic intersection using Theorem \ref%
{th:homoclinic-existence}.
\end{enumerate}

To obtain bounds for the stable/unstable manifolds, we use the local
coordinates $(x,y_{1},y_{2})$, 
\begin{equation}
\left( X,Y,Z\right) =C\left( x,y_{1},y_{2}\right) ,  \label{eq:local-coord}
\end{equation}%
with, 
\begin{equation*}
C=\left( 
\begin{array}{ccc}
1 & -0.36706363121968 & 0 \\ 
1 & 1 & 0 \\ 
0 & 0 & 1%
\end{array}%
\right) .
\end{equation*}
Note that $\left\Vert C\right\Vert \approx 1.527\,8$ and $\left\Vert
C^{-1}\right\Vert \approx 1.117\,5$.

Coordinates $x,y_{1},y_{2}$ align the system (\ref{eq:S-M-ode}) so that $x$ is the (rough) unstable
direction, and $y_{1},y_{2}$ are (roughly) stable. Note that we use the same
local coordinates for all the parameters $a$.


To obtain the bound on $W^u$ we choose 
\begin{equation*}
D=\overline{B}_{u}\left( R\right) \times \overline{B}_{s}\left( R\right) ,
\end{equation*}%
with $R=10^{-5}$, and use Theorem \ref{th:Wu-cap-bound} to obtain an
enclosure of the unstable manifold $W^{u}$. In our computer assisted proof,
we have a Lipschitz bound $L=4\cdot 10^{-5}$ for the slope of the manifold
for all parameters $a\in A$. See Figure \ref{Fig:Wu}. (Note the scale on the
axes. The enclosure is in fact quite sharp.) Applying Theorem \ref%
{th:contraction-bound} we obtain%
\begin{eqnarray*}
\xi _{1} &=&0.99999999967813, \\
c &=&2\sqrt{1+L^{2}}.
\end{eqnarray*}%
For any trajectory $\left( X(t),Y(t),Z(t)\right) $ starting from a point $%
p=C\left( x,w^{u}(x)\right), $ for $t\leq 0$, holds%
\begin{equation*}
\left\vert X\left( t\right) \right\vert ,\left\vert Y\left( t\right)
\right\vert ,\left\vert Z\left( t\right) \right\vert \leq c\left\Vert
C\right\Vert \left\Vert C^{-1}\right\Vert e^{-\xi _{1}\left\vert
t\right\vert }\left\Vert p\right\Vert <3.5e^{-\xi \left\vert t\right\vert
}\left\Vert p\right\Vert .
\end{equation*}

\begin{figure}[tbp]
\begin{center}
\includegraphics[height=6cm]{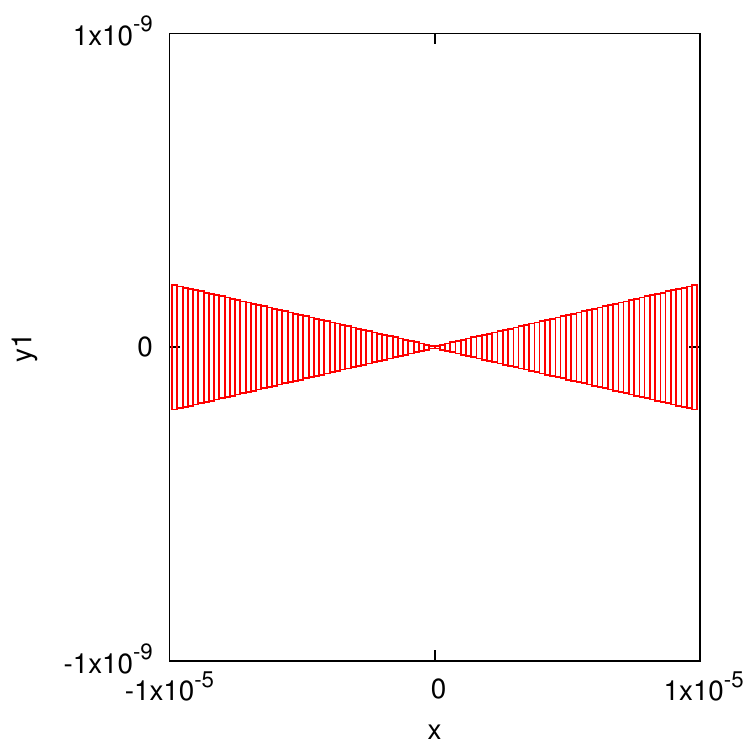}
\end{center}
\caption{The projection onto $x,y_{1}$ coordinates of the bounds on $%
W_{a}^{u}$. }
\label{Fig:Wu}
\end{figure}

To establish the bounds for the two dimensional stable manifold $W^{s}$, we
consider the vector field (\ref{eq:S-M-ode}) with reversed sign (this means that we also swap the roles of the coordinates $x$, $y_1$, $y_2$), and
apply Theorem \ref{th:Wu-cap-bound} once again. We use the same Lipschitz
bound $L$. We consider all the parameters $a\in A.$ We do so by subdividing $%
A$ into several intervals, and performing the interval arithmetic enclosure
of $W^{s}$; we use the interval enclosure of the map $f$ with the intervals
on $a$. In Figure \ref{fig:Ws} we see the bound on the enclosure. From
Theorem \ref{th:contraction-bound} we obtain%
\begin{equation*}
\xi _{2}=0.99998045374688.
\end{equation*}%
Thus, for any trajectory $\left( X(t),Y(t),Z(t)\right) $ starting from a
point $p=C\left( w^{s}(y_{1},y_{2}),y_{1},y_{2}\right) $ holds%
\begin{equation*}
\left\vert X\left( t\right) \right\vert ,\left\vert Y\left( t\right)
\right\vert ,\left\vert Z\left( t\right) \right\vert \leq c\left\Vert
C\right\Vert \left\Vert C^{-1}\right\Vert e^{-\xi _{2}t}\left\Vert
p\right\Vert <3.5e^{-\xi t}\left\Vert p\right\Vert \qquad \text{for }t\geq 0.
\end{equation*}

We now take $T=26$. The two rectangles in Figure \ref{fig:Ws} are the $\Phi
_{T}\left( Cp_{a}^{u},a\right) $ for $a=a_{l}$ and $a=a_{r}$ (see (\ref%
{eq:pu-theta-def}) for the definition of $p_{u}^{u}$).  Note
that Figure \ref{fig:Ws} corresponds to the sketch from Figure \ref%
{fig:homoclinic-2}. In Figure \ref{fig:Ws} we have the projection onto $%
x,y_{1}$ coordinates of what happens inside of the set $D$, without plotting
the trajectory along the unstable manifold. Figure \ref%
{fig:original-coordinates} presents the bounds, plotted in the original
coordinates of the system.

We use the rigorous estimates for $\Phi _{T}\left(
Cp_{a_{l}}^{u},a_{l}\right) $ and $\Phi _{T}\left(
Cp_{a_{r}}^{u},a_{r}\right) $ to compute the following bounds (see (\ref%
{eq:h-function-def}) for the definition of the function $h$,)%
\begin{eqnarray*}
h\left( a_{l}\right) & \in \lbrack 1.5093787863274e-09,3.9653443827625e-09],
\\
h\left( a_{r}\right) & \in \left[ -3.9570292285809e-09,-1.51777586447e-09%
\right] .
\end{eqnarray*}%
We also make sure that $\Phi _{T}\left( Cp_{a}^{u},a\right) \in D$ for all $%
a\in A$. We see that assumption (\ref{eq:Bolzano-assmpt}) of Theorem \ref%
{th:homoclinic-existence} is satisfied, which means that we have a
homoclinic connection for at least one of the parameters $a\in A$.

We have thus established a homoclinic orbit 
\begin{equation*}
\left( X_{0}(t),Y_{0}(t),Z_{0}(t)\right) =\Phi _{t}\left(
Cp_{a_{0}}^{u},a_{0}\right) ,
\end{equation*}%
for some $a_{0}\in A$.

\begin{figure}[tbp]
\begin{center}
\includegraphics[height=4cm]{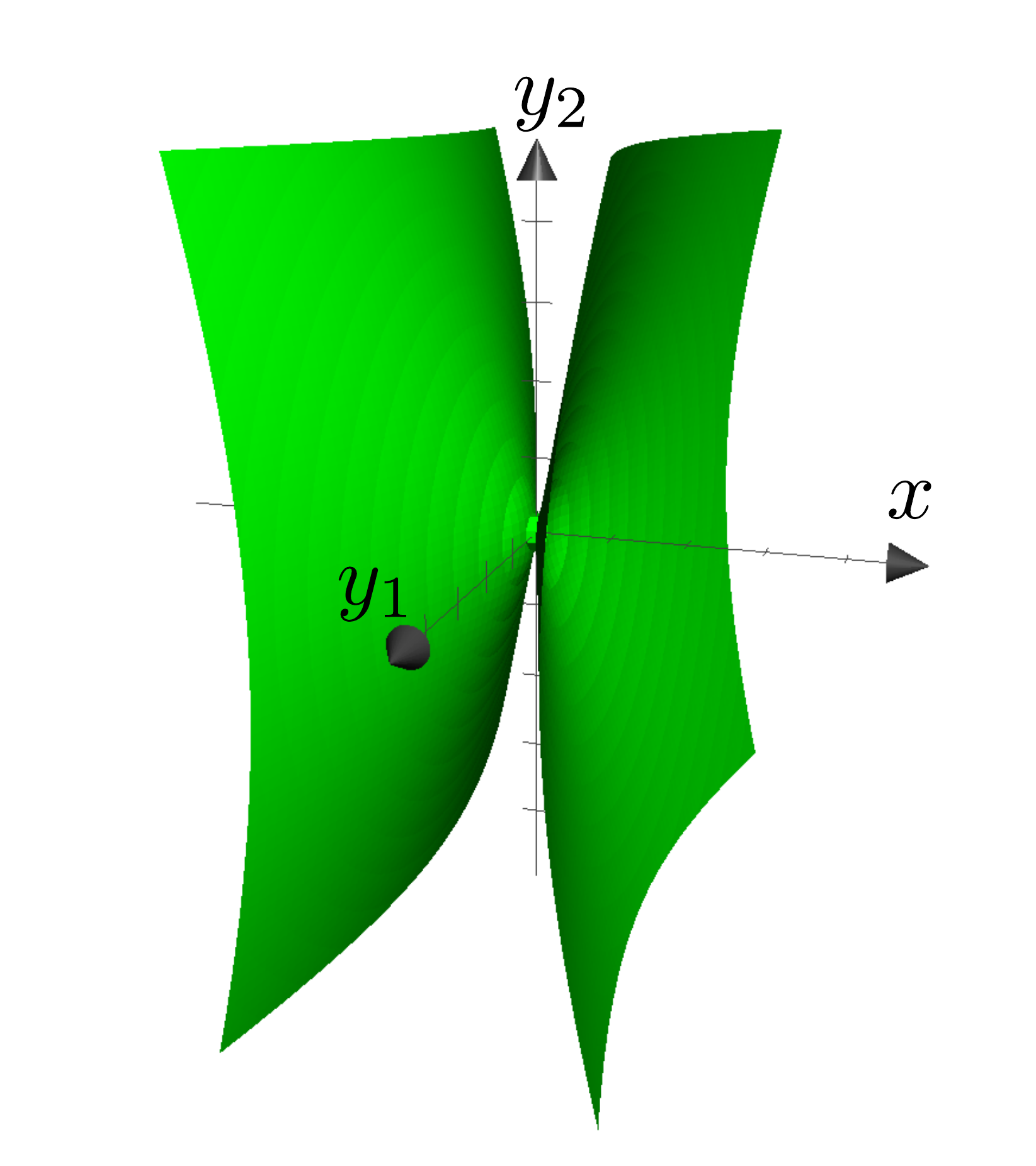}%
\includegraphics[height=4cm]{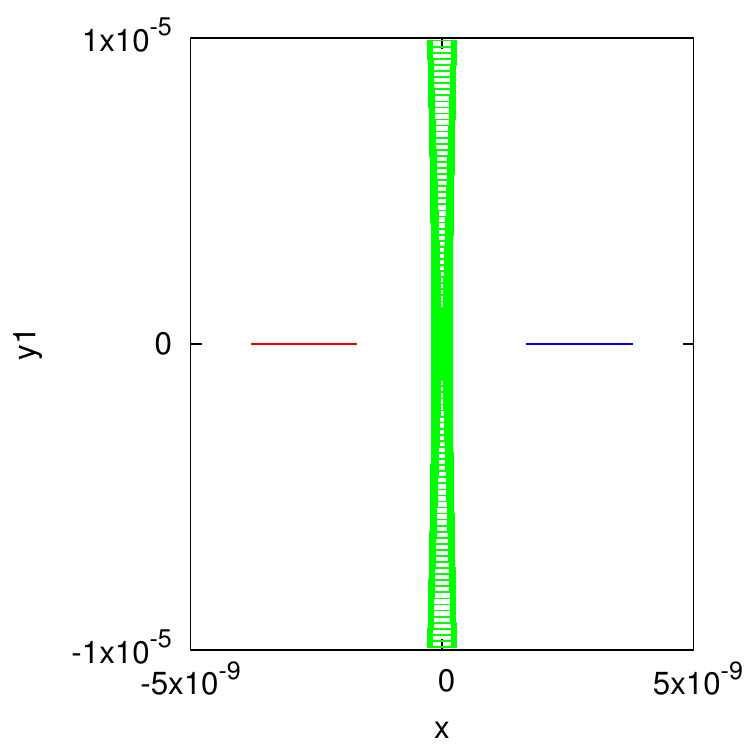} %
\includegraphics[height=4cm]{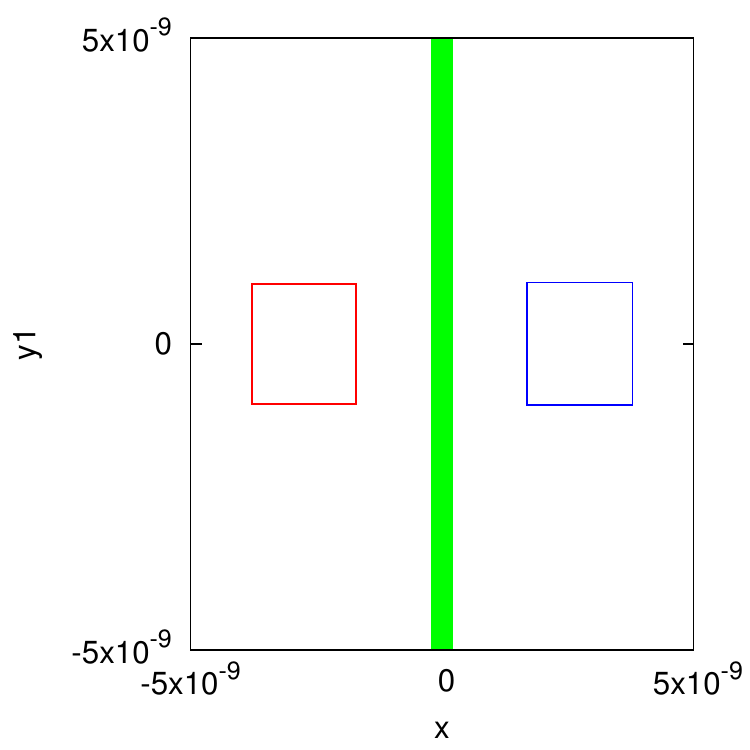}
\end{center}
\caption{The bound on $W_{a}^{s}$, for all parameters $a$. On the left we
have a non-rigorous plot, to illustrate the shape of our bound in three
dimensions. In the middle and on the right, we have a projection onto the $%
x,y_{1}$ coordinates of the rigorous, computer assisted enclosure. The two
rectangles depicted on the right hand side plots are $\Phi _{T}\left(
p_{a_{l}}^{u},a_{l}\right) $ (on the left, in red) and $\Phi _{T}\left(
p_{a_{r}}^{u},a_{r}\right) $ (on the right, in blue). All these plots are in
the local coordinates.}
\label{fig:Ws}
\end{figure}

\begin{figure}[tbp]
\begin{center}
\includegraphics[height=6cm]{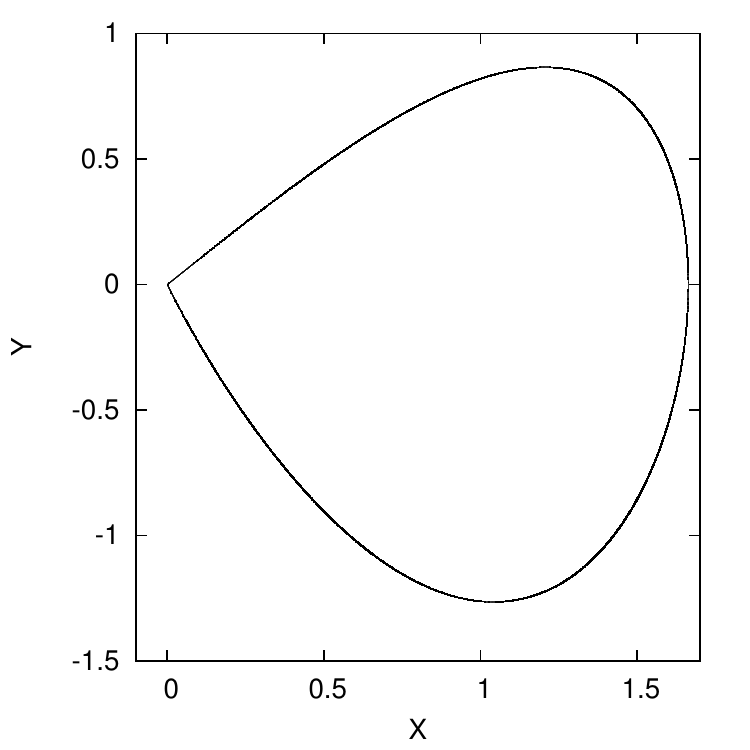}%
\includegraphics[height=6cm]{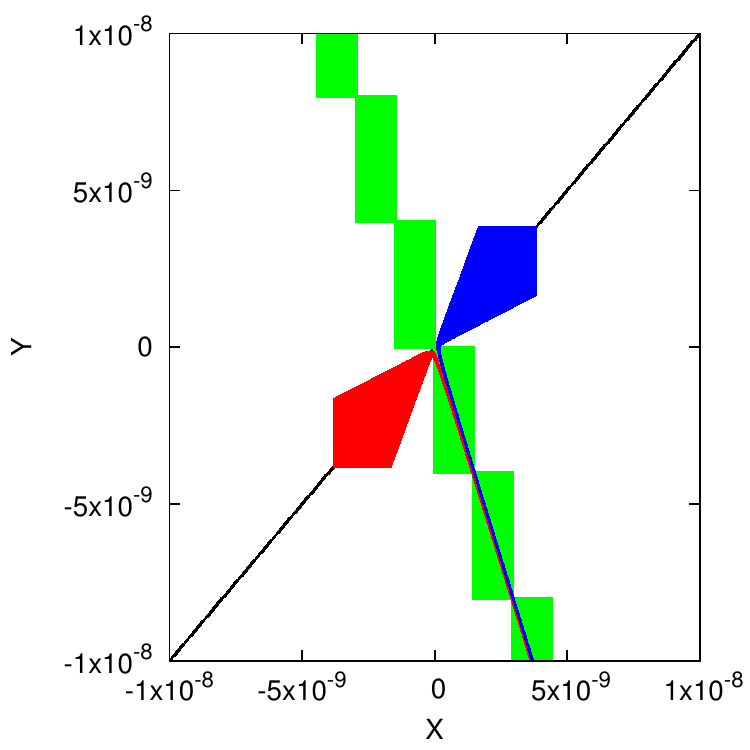}
\end{center}
\caption{The bound on $W_{a}^{u}$, for all parameters $a\in [a_l,a_r]$ on
the left (in black). On the right, we have the bound on $W_{a}^{u}$ in
black, valid for all $a\in [a_l,a_r]$. A trajectory along $W_{a}^{u}$, which
leaves the neighbourhood by the top right corner, returns through the lower
edge of the plot. In red, we have the $W_{a_l}^{u}$, as it returns to the
neighbourhood, and in blue we have $W_{a_r}^{u}$. In green is the bound on $%
W_{a}^{s}$, for all parameters $a\in [a_l,a_r]$. The bound on $W_{a}^{s}$ in
local coordinates is in fact is much tighter than on this plot. It is a
strip that passes through the intersections of the green rectangles. The
plots are in the original coordinates of the system, projected onto the $X,Y$
coordinates.}
\label{fig:original-coordinates}
\end{figure}

The computer assisted proof has been done entirely by using the CAPD%
\footnote[1]{%
computer assisted proofs in dynamics: http://capd.ii.uj.edu.pl/} package and
took under a second on a single core 3Ghz Intel i7 processor.


\section{\textbf{Computation of the separatrix value}}

\label{sec:sep-val} In this section we will show how we prove Theorem \ref%
{th:sep-val}. First we will describe the method. We shall investigate the
following ODE%
\begin{equation}
\gamma ^{\prime }\left( t\right) =\left( A+B\left( t\right) \right) \gamma
\left( t\right) ,  \label{eq:ode}
\end{equation}%
and assume that $a_{22},a_{33}>0$ and that%
\begin{equation*}
A=\left( 
\begin{array}{lll}
0 & 0 & 0 \\ 
0 & -a_{22} & 0 \\ 
0 & 0 & -a_{33}%
\end{array}%
\right) ,\qquad B\left( t\right) =\left( 
\begin{array}{lll}
b_{11}\left( t\right) & b_{12}\left( t\right) & b_{13}\left( t\right) \\ 
b_{21}\left( t\right) & b_{22}\left( t\right) & b_{23}\left( t\right) \\ 
b_{31}\left( t\right) & b_{32}\left( t\right) & b_{33}\left( t\right)%
\end{array}%
\right) .
\end{equation*}%
We assume also that there exist constants $c_{b},\lambda >0$ for which 
\begin{equation}
|b_{ij}(t)|\leq c_{b}e^{-\lambda \left\vert t\right\vert }.
\label{eq:lin-contraction-bounds}
\end{equation}

\begin{remark}
In section \ref{sec:CAP-separatrix-value}, where we will apply the method to
the Shimizu-Marioka system (\ref{eq:S-M-ode}), the constant $c_{b}$ from (%
\ref{eq:lin-contraction-bounds}) will be associated with the constant $c$
from Theorem \ref{th:homoclinic}, with the size of the neighborhood of zero
in which we investigate the rate of convergence along the homoclinic, and
also on some coordinate changes.
\end{remark}

\begin{remark}
\correction{22}{22}{Equation (\ref{eq:ode}) }can be seen as (\ref{eq:S-M-ode2}) considered in
appropriate local coordinates. This will be our approach when applying the
results from this section for the proof of Theorem \ref{th:sep-val}.
\end{remark}

In the sequel in our investigations of the problem (\ref{eq:ode}) we will
use the variables $(x,y_{1},y_{2})$ and quite often we will also write $%
y=(y_{1},y_{2})$, so that our coordinates will be $(x,y)$, where $x\in 
\mathbb{R}$ and $y\in \mathbb{R}^{2}$.

Our objective will be to obtain a proof of an orbit $\gamma ^{\ast }\left(
t\right) $ of (\ref{eq:ode}) for which%
\begin{eqnarray}
\lim_{t\rightarrow -\infty }\gamma ^{\ast }\left( t\right) & =(x_{-}^{\ast
},0,0),  \label{eq:x-minus-def} \\
\lim_{t\rightarrow +\infty }\gamma ^{\ast }\left( t\right) & =(x_{+}^{\ast
},0,0),  \label{eq:x-plus-def}
\end{eqnarray}%
and also to obtain explicit bounds on the fraction $\frac{x_{+}^{\ast }}{%
x_{-}^{\ast }}$.

To understand the behavior of (\ref{eq:ode}) for large $|t|$, we treat the
line $\{(x,0,0),\,x\in \mathbb{R}\}$ as the normally hyperbolic manifold,
with two stable directions ($y_{1}$ and $y_{2}$). This is clearly visible in (%
\ref{eq:ode}) with $B\equiv 0$, but for $|t|$ large $B$ is just a small
perturbation so the normally hyperbolic behavior will survive. A bit
problematic in this picture is the fact that we have here an non-autonomous
system, so the above scenario needs some adjustments. This is the idea,
which underlines our approach.

\begin{definition}
\label{def:LapFunc} Let $Z\subset \mathbb{R}\times \mathbb{R}^{2}$ and $%
J\subset \mathbb{R}$. Let $V:Z\rightarrow \mathbb{R}$ be a $C^{1}$ function.
We say that $V$ is a Lyapunov function on $J\times Z$ for (\ref{eq:ode}) if
for any solution $z:J\rightarrow \mathbb{R}\times \mathbb{R}^{2}$ of (\ref%
{eq:ode}), $V(z(t))$ is decreasing ($\frac{d}{dt}V(z(t))<0$) for $z(t)\in Z$
and $t\in J$.
\end{definition}

We have 
\begin{equation}
\frac{d}{dt}V(z(t))=((A+B(t))z(t)|\nabla V(z(t))).  \label{eq:derAlongSol}
\end{equation}

We start with two technical lemmas: 
\begin{figure}[tbp]
\begin{center}
\includegraphics[height=5cm]{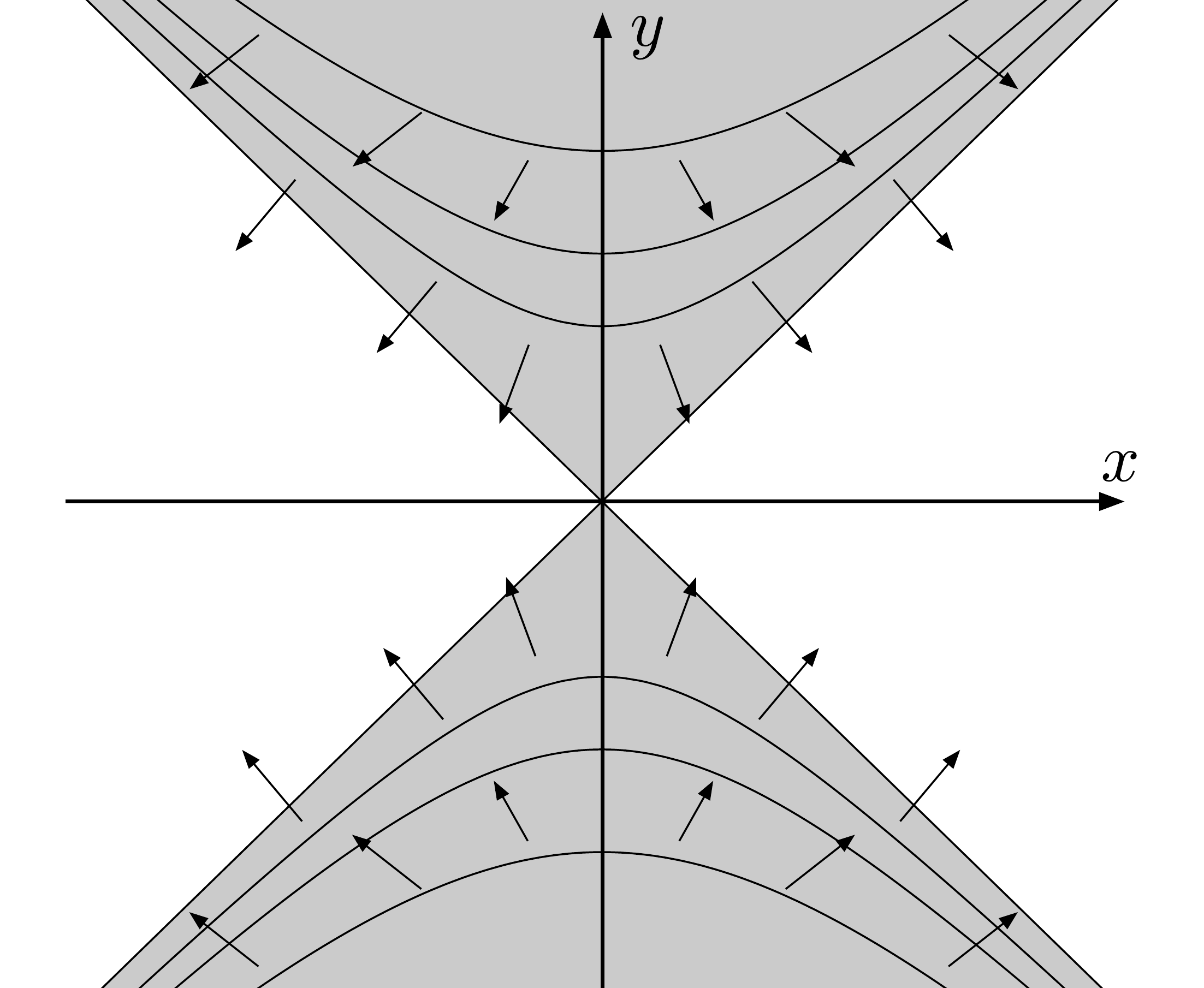}
\end{center}
\caption{The set $D$ from Lemma \protect\ref{lem:Lap-f2} is the vertical
cone. The curved lines are the level sets of $V$. The arrows indicate the
vector field of (\protect\ref{eq:ode}).}
\label{fig:Lyap1}
\end{figure}

\begin{lemma}
\label{lem:Lap-f2} Let $D=\left\{ \left\Vert y\right\Vert ^{2}\geq
x^{2}\right\} \setminus \{0\}$ (see Figure \ref{fig:Lyap1}). Let $J\subset 
\mathbb{R}$ be a set of all $t\in \mathbb{R}$, for which the following
condition holds 
\begin{equation}
6c_{b}e^{-\lambda \left\vert t\right\vert }<\min \left( a_{22},a_{33}\right)
.  \label{eq:t-cond1}
\end{equation}%
Then \correction{23}{23}{
\begin{eqnarray*}
V& :\mathbb{R\times R}^{2}\rightarrow \mathbb{R}, \\
V\left( x,y\right) & =\left\Vert y\right\Vert ^{2}-x^{2}
\end{eqnarray*}}
is a Lyapunov function for (\ref{eq:ode}) on $J\times D$.
\end{lemma}

\begin{proof}
For $p=\left( x,y_{1},y_{2}\right) $ we compute (skipping the dependence of $%
b_{ij}(t)$ on $t$ to simplify the notation): 
\begin{eqnarray*}
&&\left( \left( A+B\left( t\right) \right) p|\frac{1}{2}\nabla V\left(
p\right) \right) \\
&=&-x^{2}b_{11}+y_{1}^{2}\left( b_{22}-a_{22}\right) +y_{2}^{2}\left(
b_{33}-a_{33}\right) \\
&&+xy_{1}\left( b_{21}-b_{12}\right) +xy_{2}\left( b_{31}-b_{13}\right)
+y_{1}y_{2}\left( b_{23}+b_{32}\right) \\
&\leq &\left\Vert y\right\Vert ^{2}\left\vert b_{11}\right\vert +\left(
y_{1}^{2}+y_{2}^{2}\right) \left( \max \left( \left\vert b_{22}\right\vert
,\left\vert b_{33}\right\vert \right) +\max \left( -a_{22},-a_{33}\right)
\right) \\
&&+\frac{1}{2}\left( x^{2}+y_{1}^{2}\right) \left( \left\vert
b_{21}\right\vert +\left\vert b_{12}\right\vert \right) +\frac{1}{2}\left(
x^{2}+y_{2}^{2}\right) \left( \left\vert b_{31}\right\vert +\left\vert
b_{13}\right\vert \right) \\
&&\qquad +\frac{1}{2}\left( y_{1}^{2}+y_{2}^{2}\right) \left( \left\vert
b_{23}\right\vert +\left\vert b_{32}\right\vert \right) \\
&\leq &\left[ 2c_{b}e^{-\lambda \left\vert t\right\vert }+\max \left(
-a_{22},-a_{33}\right) \right] \left\Vert y\right\Vert ^{2} \\
&&+c_{b}e^{-\lambda \left\vert t\right\vert }\left( \left(
x^{2}+y_{1}^{2}\right) +\left( x^{2}+y_{2}^{2}\right) +\left(
y_{1}^{2}+y_{2}^{2}\right) \right) \\
&\leq &\left[ 2c_{b}e^{-\lambda \left\vert t\right\vert }+\max \left(
-a_{22},-a_{33}\right) \right] \left\Vert y\right\Vert ^{2}+c_{b}e^{-\lambda
\left\vert t\right\vert }\left( 4\left\Vert y\right\Vert ^{2}\right) \\
&=&\left[ 6c_{b}e^{-\lambda \left\vert t\right\vert }-\min \left(
a_{22},a_{33}\right) \right] \left\Vert y\right\Vert ^{2} \\
&<&0,
\end{eqnarray*}%
which concludes the proof.
\end{proof}

\begin{figure}[tbp]
\begin{center}
\includegraphics[height=5cm]{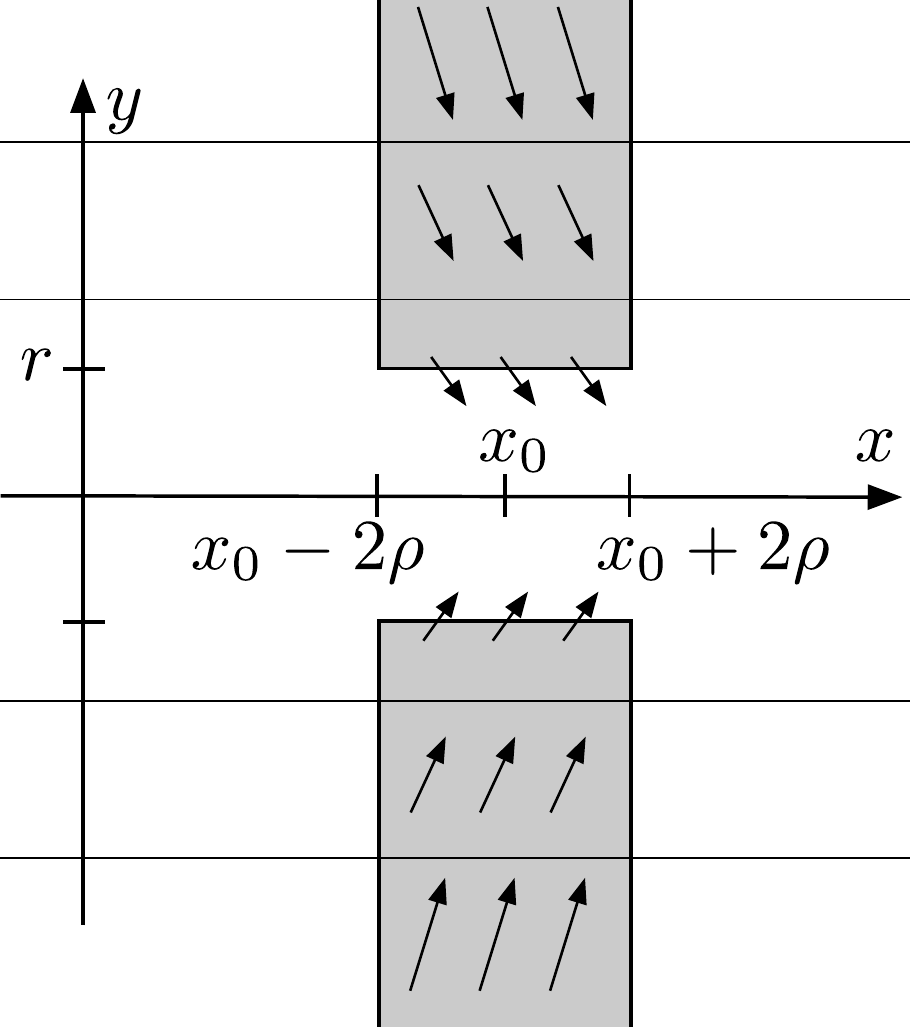}
\end{center}
\caption{The set $D$ from Lemma \protect\ref{lem:Lap-f} is represented by
the two shaded rectangular strips. The horizontal lines are the level sets
of $V$. The arrows indicate the vector field of (\protect\ref{eq:ode}).}
\label{fig:Lyap2}
\end{figure}

\begin{lemma}
\label{lem:Lap-f}Let $x_{0},\rho ,r>0$ and let%
\begin{equation*}
D=\left[ x_{0}-2\rho ,x_{0}+2\rho \right] \times \left\{ \left\Vert
y\right\Vert \geq r\right\} \subset \mathbb{R\times R}^{2}.
\end{equation*}%
Let $J$ be a set of $t\in \mathbb{R}$, such that 
\begin{equation}
2c_{b}e^{-\lambda \left\vert t\right\vert }\left( 1+\frac{x_{0}+2\rho }{r}%
\right) <\min \left( a_{22},a_{33}\right) ,  \label{eq:t-cond}
\end{equation}%
then \correction{24}{24}{
\begin{eqnarray*}
V& :\mathbb{R\times R}^{2}\rightarrow \mathbb{R}, \\
V\left( x,y\right) & =\left\Vert y\right\Vert ^{2}
\end{eqnarray*}}
is a Lyapunov function for (\ref{eq:ode}) on the set $J\times D$.
\end{lemma}

\begin{proof}
Let $p=\left( x,y_{1},y_{2}\right) $ and let $R^{2}=y_{1}^{2}+y_{2}^{2}\geq
r $. Note that $\pm y_{1}y_{2}\leq \frac{1}{2}R^{2}$ and $\left\vert
y_{1}\right\vert ,\left\vert y_{2}\right\vert \leq R$. We can compute
(skipping the dependence of $b_{ij}(t)$ on $t$ to simplify the notation) 
\begin{eqnarray*}
& \left( \left( A+B\left( t\right) \right) p|\frac{1}{2}\nabla V\left(
p\right) \right) \\
& =\left( -a_{22}+b_{22}\right) y_{1}^{2}+\left( -a_{33}+b_{33}\right)
y_{2}^{2} \\
& \quad +\left( b_{23}+b_{32}\right) y_{1}y_{2}+b_{21}xy_{1}+b_{31}xy_{2} \\
& \leq \max \left( -a_{22}+b_{22},-a_{33}+b_{33}\right) R^{2} \\
& \quad +\frac{1}{2}\left( \left\vert b_{23}\right\vert +\left\vert
b_{32}\right\vert \right) R^{2}+\left\vert b_{21}\right\vert \left\vert
x\right\vert R+\left\vert b_{31}\right\vert \left\vert x\right\vert R \\
& \leq \left( \max \left( -a_{22},-a_{33}\right) +2c_{b}e^{-\lambda
\left\vert t\right\vert }\right) R^{2}+2c_{b}e^{-\lambda \left\vert
t\right\vert }\left( x_{0}+2\rho \right) R \\
& \leq \left[ \max \left( -a_{22},-a_{33}\right) +2c_{b}e^{-\lambda
\left\vert t\right\vert }\left( 1+\frac{x_{0}+2\rho }{r}\right) \right] R^{2}
\\
& <0,
\end{eqnarray*}%
as required.
\end{proof}

Let $\gamma\left( t\right) $ be the solution of (\ref{eq:ode}) with initial
condition $\gamma\left( t_{0}\right) =p$. We shall define%
\begin{equation*}
\phi_{t_{0},t}\left( p\right) :=\gamma\left( t_{0}+t\right) .
\end{equation*}
Note that%
\begin{equation}
\frac{d}{dt}\phi_{t_{0},t}\left( p\right) =\left( A+B\left( t_{0}+t\right)
\right) \phi_{t_{0},t}\left( p\right) .  \label{eq:d-phi}
\end{equation}

\begin{definition}
\label{def:Uxrho} Let $r>0$, $a\geq 0$ and $x \in \mathbb{R}$ we define sets 
$U_{x,a},U_{x,a}^{+},U_{x,a}^{-}\subset \mathbb{R}\times \mathbb{R}^{2}$ by
(see Figure \ref{fig:SetsU})%
\begin{eqnarray*}
U_{x,a}& =\left[ x-a,x+a\right] \times \overline{B}(0,r), \\
U_{x,a}^{+}& =\left[ x-a,x+a\right] \times \partial \overline{B}(0,r), \\
U_{x,a}^{-}& =\partial \left[ x-a ,x+a \right] \times \overline{B}(0,r).
\end{eqnarray*}
\end{definition}

\begin{figure}[tbp]
\begin{center}
\includegraphics[height=5cm]{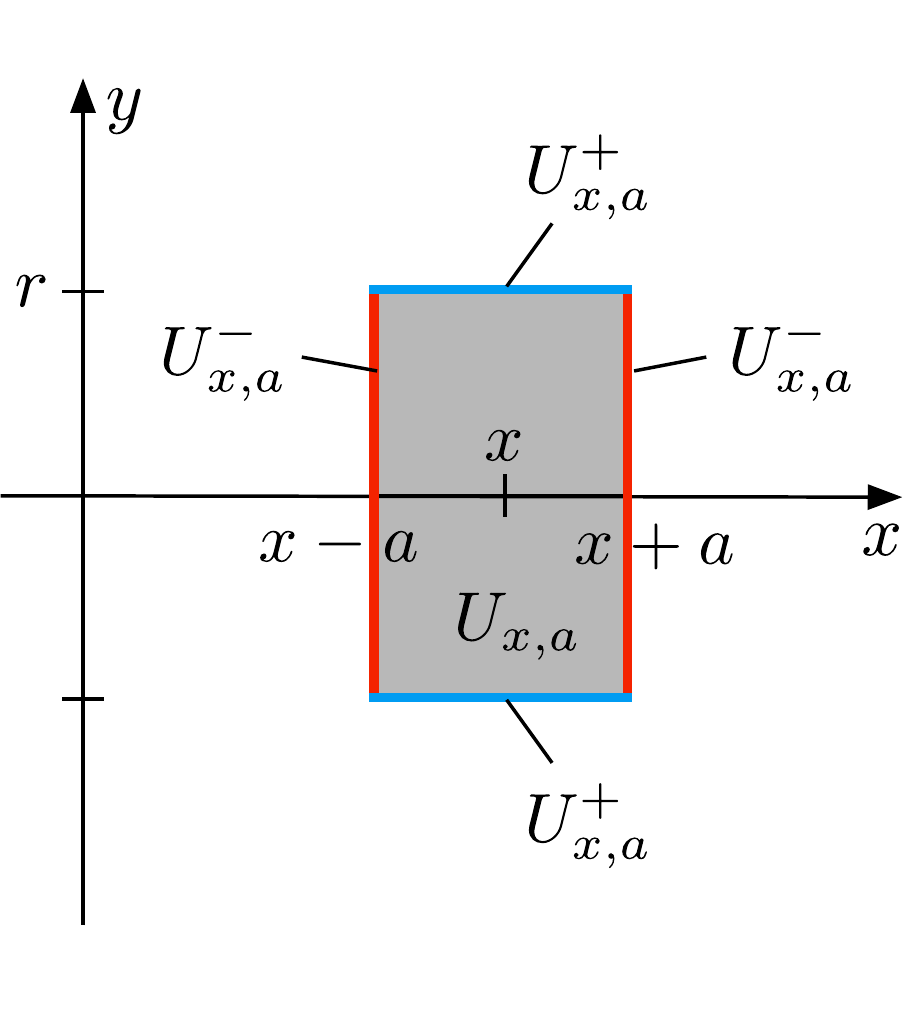}
\end{center}
\caption{The sets $U_{x,a}$, $U_{x,a}^{+}$ and $U_{x,a}^{-}$.}
\label{fig:SetsU}
\end{figure}

\subsection{Behavior for $t \to -\infty$}

\label{subsec:sep-val-minus-infty}

Below lemma allows us to obtain a bound on $x_{-}^{\ast }$ from (\ref%
{eq:x-minus-def}). The idea is to choose some $x^{\ast }\in \mathbb{R}$ and
an initial condition in $U_{x^{\ast },0}$ at an initial time $t^{\ast }<0$.
If $t^{\ast }$ is small enough, then the influence of the matrix $B\left(
t\right) $ on (\ref{eq:ode}) will be small.

\begin{lemma}
\label{lem:backward-sep-val}Let us fix $x^{\ast }>0$, $\rho >0$, $r>0$ and $%
t^{\ast }<0$, such that conditions (\ref{eq:t-cond1}), (\ref{eq:t-cond})
hold true for $x_{0}=x^{\ast }$. If 
\begin{equation}
\frac{1}{\lambda }c_{b}e^{-\lambda \left\vert t^{\ast }\right\vert }\left(
x^{\ast }+2\rho +2r\right) <\rho ,  \label{eq:rho-t-r-link}
\end{equation}%
then there exists a unique point $p^{\ast }\in U_{x^{\ast },0}$ such that $%
\phi _{t^{\ast },t}\left( p^{\ast }\right) $ is convergent as $t\rightarrow
-\infty $. Moreover, 
\begin{equation*}
\lim_{t\rightarrow -\infty }\phi _{t^{\ast },t}\left( p^{\ast }\right)
=\left( x_{-}^{\ast },0,0\right) \in U_{x^{\ast },\rho }.
\end{equation*}
\end{lemma}

\begin{proof}
Since (\ref{eq:t-cond}) holds for $t^{\ast }$, by Lemma \ref{lem:Lap-f} we
see that $U_{x^{\ast },2\rho }^{+}$ is an \textquotedblleft entry set" for
the set $U_{x^{\ast },2\rho }$ (see Figures \ref{fig:Lyap2} ,\ref{fig:SetsU}%
). What we mean by this statement is that for $p\in U_{x^{\ast },2\rho }$
and any $t_{0}<t^{\ast }$ the trajectory $\phi _{t_{0},s}(p)$ can not leave
the set $U_{x^{\ast },2\rho }$ by passing through $U_{x^{\ast },2\rho }^{+}$
going forwards in time for $s\in \left[ 0,t^{\ast }-t_{0}\right] $. The only
way it can exit is by passing through $U_{x^{\ast },2\rho }^{-}$.

We will now show that for $p\in U_{x^{\ast },\rho }$ and any $t_{0}<t^{\ast
} $, the trajectory $\phi _{t_{0},s}(p)$ will never leave $U_{x^{\ast
},2\rho } $ for $s\in \left[ 0,t^{\ast }-t_{0}\right] $. Since we can not
exit through $U_{x^{\ast },2\rho }^{+}$, we need to take care so that for $%
s\in \left[ 0,t^{\ast }-t_{0}\right] $ the $\phi _{t_{0},s}(p)$ does not
reach $U_{x^{\ast },2\rho }^{-}$. Let us use the notation%
\begin{equation*}
\phi _{t_{0},s}\left( p\right) =\left( x\left( s\right) ,y_{1}\left(
s\right) ,y_{2}\left( s\right) \right) ,
\end{equation*}%
As long as $\left\vert \pi _{x}\phi _{t_{0},t}\left( p\right) -x^{\ast
}\right\vert \leq 2\rho $ for $t\in \left[ 0,s\right] $, we have the
following bound (using (\ref{eq:rho-t-r-link}) in the last inequality) 
\begin{eqnarray}
& \left\vert \pi _{x}\phi _{t_{0},s}\left( p\right) -x^{\ast }\right\vert 
\nonumber \\
& \leq \left\vert \pi _{x}p-x^{\ast }\right\vert +\left\vert
\int_{0}^{s}b_{11}\left( t_{0}+t\right) x\left( t\right) +b_{12}\left(
t_{0}+t\right) y_{1}\left( t\right) +b_{13}\left( t_{0}+t\right) y_{2}\left(
t\right) dt\right\vert  \nonumber \\
& \leq \rho +\int_{-\infty }^{t^{\ast }-t_{0}}c_{b}e^{\lambda \left(
t_{0}+t\right) }\left( x^{\ast }+2\rho +2r\right) dt  \nonumber \\
& =\rho +\frac{1}{\lambda }c_{b}e^{\lambda t^{\ast }}\left( x^{\ast }+2\rho
+2r\right)  \nonumber \\
& <\hat{c}2\rho ,  \label{eq:cdelta-bound}
\end{eqnarray}%
for a fixed $\hat{c}\in \left( \frac{1}{2},1\right) $, which does not depend
on $s$. 
\begin{equation*}
\forall t_{0}\leq t^{\ast },\ \forall s\in \lbrack 0,t^{\ast }-t_{0}],\qquad
\quad \Vert \pi _{x}\phi _{t_{0},s}(p)-x^{\ast }\Vert \leq 2\rho .
\end{equation*}

This shows that $\phi _{t_{0},s}$ will not leave $U_{x^{\ast },2\rho }.$
Since trajectories from $U_{x^{\ast },\rho }$ do not leave $U_{x^{\ast
},2\rho }$ and can not touch $U_{x^{\ast },2\rho }^{+}$, we have established
that 
\begin{equation}
\forall t_{0}<t^{\ast },\ \forall s\in (0,t^{\ast }-t_{0}],\qquad \phi
_{t_{0},s}(U_{x^{\ast },\rho })\subset \mathrm{int}U_{x^{\ast },2\rho }.
\label{eq:covering1}
\end{equation}

Mirror estimates to (\ref{eq:cdelta-bound}) together with (\ref%
{eq:rho-t-r-link}) lead to%
\begin{equation*}
|\pi _{x}\phi _{t_{0},s}(p)-\pi _{x}p|\leq \frac{1}{\lambda }c_{b}e^{\lambda
t^{\ast }}(x^{\ast }+2\rho +2r)<\rho ,
\end{equation*}%
hence we proved that 
\begin{equation}
|\pi _{x}\phi _{t_{0},s}(p)-\pi _{x}p|<\rho ,\quad \forall t_{0}<t^{\ast
}\,\forall p\in U_{x^{\ast },\rho }\,\forall s\in \lbrack 0,t^{\ast }-t_{0}].
\label{eq:shift-on-x}
\end{equation}

In particular we obtain for any $y\in \overline{B(0,r)}$, for all $t_{0}\leq
t^{\ast }$ and $s\in \lbrack 0,t^{\ast }-t_{0}]$ 
\begin{equation}
\pi _{x}\phi _{t_{0},s}(x^{\ast }-\rho ,y)<x^{\ast },\quad \pi _{x}\phi
_{t_{0},s}(x^{\ast }+\rho ,y)>x^{\ast }.  \label{eq:on-diff-sides}
\end{equation}%
\begin{figure}[tbp]
\begin{center}
\includegraphics[height=5cm]{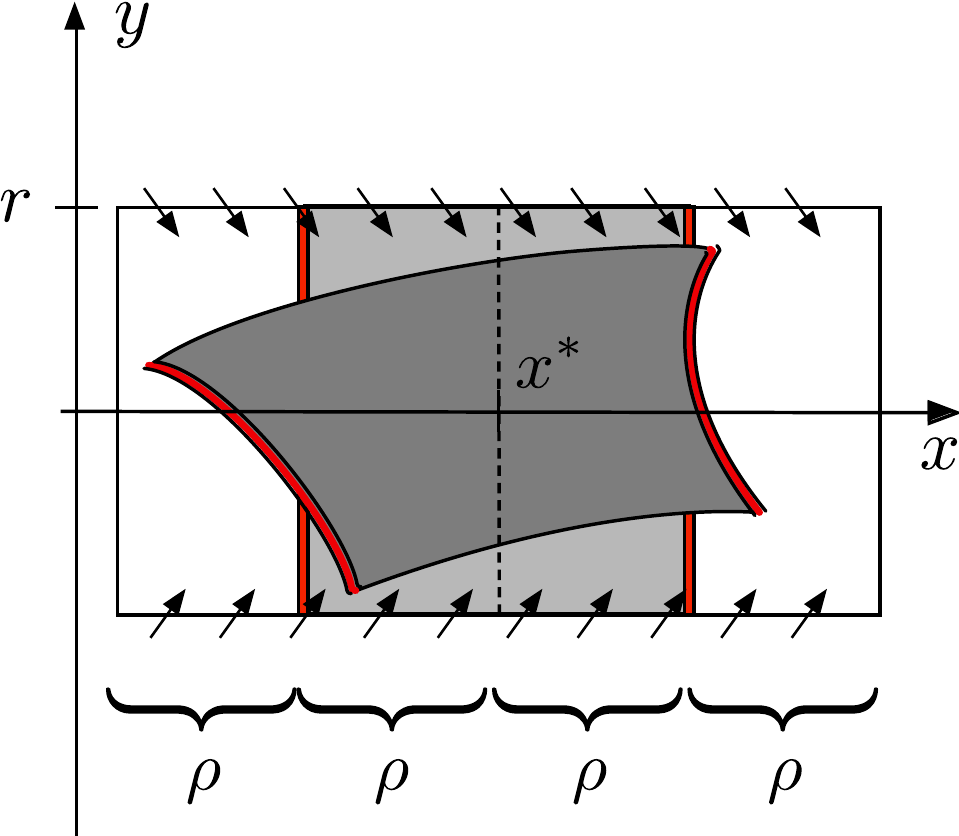}
\end{center}
\caption{The set $U_{x^{\ast },\rho }$ in light grey and 
$\phi
_{t_{0},s}\left( U_{x^{\ast },\rho }\right) $ in dark grey. The set $U_{x^{\ast },0}$ is represented by the dashed vertical line.}
\label{fig:topol}
\end{figure}
From (\ref{eq:covering1}) and (\ref{eq:on-diff-sides}), due to the
topological alignment (see Figure \ref{fig:topol}) of the sets $\phi
_{t_{0},s}\left( U_{x^{\ast },\rho }\right) $ and $U_{x^{\ast },0}$, it
follows that for any $t_{0}<t^{\ast }$ there exist points $q=q\left(
t_{0}\right) \in U_{x^{\ast },\rho }$ and $p=p\left( t_{0}\right) \in
U_{x^{\ast },0}$, such that $\phi _{t_{0},t^{\ast }-t_{0}}\left( q\right) =p$
and $\phi _{t_{0},s}\left( q\right) \in U_{x^{\ast },2\rho }$ for $s\in %
\left[ 0,t^{\ast }-t_{0}\right] $. This implies that 
\begin{eqnarray}
\phi _{t^{\ast },t_{0}-t^{\ast }}\left( p\right) & \in U_{x^{\ast },\rho },
\label{eq:Urho-bound} \\
\phi _{t^{\ast },s}\left( p\right) & \in U_{x^{\ast },2\rho }\qquad \text{%
for }s\in \lbrack t_{0}-t^{\ast },0].  \nonumber
\end{eqnarray}

Let $p_{n}=p\left( -n\right) $ (by taking $t_{0}=-n$).

Since $U_{x^{\ast },0}$ is compact, we can pass to a convergent subsequence \correction{25}{25}{}$
p_{n_{k}}\rightarrow p^{\ast }\in U_{x^{\ast },0}$, and obtain a point for
which 
\begin{equation}
\phi _{t^{\ast },s}\left( p^{\ast }\right) \in U_{x^{\ast },2\rho }\qquad 
\text{for }s\in (-\infty ,0].  \label{eq:U-2rho-holds-trajectory}
\end{equation}%
We now need to show that $\phi _{t^{\ast },s}\left( p^{\ast }\right) $ is
convergent as $s\rightarrow -\infty $.

We first focus on the $x$ coordinate. We will show that for any $\varepsilon 
$, there exists an $S<0$, such that for any $s_{1},s_{2}<S$ holds (the
Cauchy condition for the function $s\mapsto \pi _{x}\phi _{t^{\ast
},s_{1}}\left( p^{\ast }\right) $ for $s\rightarrow -\infty $) 
\begin{equation}
\left\Vert \pi _{x}\phi _{t^{\ast },s_{1}}\left( p^{\ast }\right) -\pi
_{x}\phi _{t^{\ast },s_{2}}\left( p^{\ast }\right) \right\Vert \leq
\varepsilon .  \label{eq:x-convergence-b}
\end{equation}%
Assume that $s_{1}<s_{2}$ and that $u=s_{1}-s_{2}$. Then, by estimates
analogous to (\ref{eq:cdelta-bound}),%
\begin{eqnarray*}
& \left\Vert \pi _{x}\phi _{t^{\ast },s_{1}}\left( p^{\ast }\right) -\pi
_{x}\phi _{t^{\ast },s_{2}}\left( p^{\ast }\right) \right\Vert \\
& =\left\Vert \pi _{x}\phi _{t^{\ast },s_{2}+u}\left( p^{\ast }\right) -\pi
_{x}\phi _{t^{\ast },s_{2}}\left( p^{\ast }\right) \right\Vert \\
& \leq \frac{1}{\lambda }c_{b}e^{s_{2}}\left( x^{\ast }+2\rho +2r\right) .
\end{eqnarray*}%
We see that by choosing negative $s_{2}$, with $\left\vert s_{2}\right\vert $
sufficiently large, we obtain (\ref{eq:x-convergence-b}). This implies the
existence of $\lim_{s\rightarrow -\infty }\pi _{x}\phi _{t^{\ast
},s_{1}}\left( p^{\ast }\right) .$

We will now show that 
\begin{equation}
\lim_{s\rightarrow -\infty }\pi _{y}\phi _{t^{\ast },s}\left( p^{\ast
}\right) =0.  \label{eq:convergence-y}
\end{equation}
Suppose that this is not the case. Thus, there would exist a $\delta $ and a
sequence $s_{i}\rightarrow -\infty $ such that 
\begin{equation*}
\left\vert \pi _{y}\phi _{t^{\ast },s_{i}}\left( p^{\ast }\right)
\right\vert >\delta .
\end{equation*}%
Let us choose $s_{i}$ negative enough so that the $V\left( x,y\right)
=\left\Vert y\right\Vert ^{2}$ would be a Lyapunov function on $J\times D$
with $J=(-\infty, s_i]$ and 
\begin{equation*}
D=\left[ x^{\ast }-2\rho ,x^{\ast }+2\rho \right] \times \left\{ \left\Vert
y\right\Vert \geq \frac{\delta }{2}\right\} .
\end{equation*}
This is possible, since we can apply Lemma \ref{lem:Lap-f} provided that
condition (\ref{eq:t-cond}) is satisfied. The (\ref{eq:t-cond}) will hold
for $r=\frac{\delta }{2}$, provided that $t<T$, for $T$ negative enough.
Since $V$ is a Lyapunov function, by going backward in time, for $s<s_{i}<T$%
, $\phi _{t^{\ast },s}$ will exit $U_{x^{\ast },2\rho }$. This contradicts (%
\ref{eq:U-2rho-holds-trajectory}), hence 
\begin{equation*}
\lim_{s\rightarrow -\infty }\pi _{y}\phi _{t^{\ast },s}\left( p^{\ast
}\right) =0.
\end{equation*}%
We have shown that $\phi _{t^{\ast },s}\left( p^{\ast }\right) $ is
convergent as $s\rightarrow -\infty $, and that the limit lies in $%
U_{x^{\ast },2\rho }$. The fact that the limit is in fact in $U_{x^{\ast
},\rho }$ follows from (\ref{eq:shift-on-x}).

We will now show that the point $p^{\ast }$ is unique. Suppose that we have
a second point $p^{\ast \ast }\in U_{x^{\ast },0}$ for which $%
\lim_{s\rightarrow -\infty }\phi _{t^{\ast },s}\left( p^{\ast \ast }\right) $
exists. Consider the following time dependent change of coordinates $%
(t^{\ast }+t,\xi =(t^{\ast }+t,x,y)-\phi _{t^{\ast },t}\left( p^{\ast
}\right) )$, which means that we just continuously change the location of
the origin, by subtracting some solution. Our equation is linear, hence in
new coordinates the equation is the same, $\xi ^{\prime }=(A+B(t))\xi $.

Let us take an initial condition $\xi \left( t^{\ast }\right) =p^{\ast \ast
}-p^{\ast }$. If $p^{\ast \ast }\neq p^{\ast }$, since by construction $%
p^{\ast \ast },p^{\ast }\in U_{x^{\ast },0},$ we see that $\xi \left(
t^{\ast }\right) \in D=\left\{ \left\Vert y\right\Vert ^{2}>\left\vert
x\right\vert \right\} $. Since we assumed that for $t^{\ast }$ condition (%
\ref{eq:t-cond1}) holds, it also holds for all $t<t^{\ast }$. By Lemma \ref%
{lem:Lap-f2}, this means that for $s<0,$ $\xi \left( t^{\ast }+s\right) $
stays inside $D$. Moreover, since $V(x,y)=\left\Vert y\right\Vert ^{2}-x$ is
a Lyapunov function, $\xi \left( t^{\ast }+s\right) $ must tend to infinity
as $s\rightarrow -\infty $. Since $\xi \left( t^{\ast }+s\right) =\phi
_{t^{\ast },s}\left( p^{\ast \ast }\right) -\phi _{t^{\ast },s}\left(
p^{\ast }\right) $, this means that we can not have both $p^{\ast \ast }\neq
p^{\ast }$ and $\phi _{t^{\ast },s}\left( p^{\ast \ast }\right) $
convergent. We have thus shown that $p^{\ast }$ is the only point in $%
U_{x^{\ast },0}$ for which the limit exists.
\end{proof}

\begin{corollary}
\label{col:cond-to-satisfy} If we consider $t^{\ast }=0$, $x^{\ast }=1$, $%
r=\rho $, and when%
\begin{eqnarray*}
0 &<&\lambda -4c_{b}, \\
0 &<&\min \left( a_{22},a_{33}\right) -6c_{b}, \\
r &>&\max \left( \frac{c_{b}}{\lambda -4c_{b}},\frac{2c_{b}}{\min \left(
a_{22},a_{33}\right) -6c_{b}}\right) ,
\end{eqnarray*}%
then assumptions of Lemma \ref{lem:backward-sep-val} are satisfied.
\end{corollary}

\subsection{Behavior for $t \to \infty$}

\label{subsec:sep-val-plus-infty}

We now consider a point in $U_{x^{\ast \ast },0}$ with initial time $t^{\ast
\ast }>0$.

\begin{lemma}
\label{lem:forward-sep-val} Let us fix $x^{\ast \ast }>0$, $\rho >0$, $r>0$
and $t^{\ast \ast }>0$, such that conditions (\ref{eq:t-cond1}), (\ref%
{eq:t-cond}) hold true for $x_{0}=x^{\ast \ast }$.

If 
\begin{equation*}
\frac{1}{\lambda }c_{b}e^{-\lambda t^{\ast \ast }}\left( x^{\ast \ast
}+2\rho +2r\right) <\rho ,
\end{equation*}%
then for any $p^{\ast \ast }\in U_{x^{\ast \ast },0}$ the $\phi _{t^{\ast
\ast },t}\left( p^{\ast \ast }\right) $ is convergent as $t\rightarrow
+\infty $. Moreover, 
\begin{equation}
\lim_{t\rightarrow +\infty }\phi _{t^{\ast \ast },t}\left( p^{\ast \ast
}\right) =\left( x_{+}^{\ast \ast },0,0\right) \in U_{x^{\ast \ast },\rho }.
\label{eq:forward-limit}
\end{equation}
\end{lemma}

\begin{proof}
The proof follows analogous in steps as the proof of Lemma \ref%
{lem:backward-sep-val}. Since (\ref{eq:t-cond}) holds for $t^{\ast }$, by
Lemma \ref{lem:Lap-f} we see that $U_{x^{\ast \ast },2\rho }^{+}$ is an
\textquotedblleft entry set" for the set $U_{x^{\ast \ast },2\rho }$ (see
Figures \ref{fig:Lyap2} ,\ref{fig:SetsU}). By a similar derivation to (\ref%
{eq:cdelta-bound}) we can establish that trajectories from $U_{x^{\ast },0}$
do not leave $U_{x^{\ast },\rho }$,\correction{26}{26}{
\begin{equation*}
\phi _{t^{\ast \ast },s}(U_{x^{\ast },0})\subset U_{x^{\ast \ast },\rho
}\qquad \text{for all }s>0.
\end{equation*}}
Using the same method (but taking the limit to $+\infty $ instead of $%
-\infty $) as for the proof of (\ref{eq:x-convergence-b}) we can show that
we have convergence of $\pi _{x}\phi _{t^{\ast \ast },t}\left( p^{\ast \ast
}\right) $ as $t$ goes to $+\infty $ for any $p^{\ast \ast }\in U_{x^{\ast
\ast },0}$.

We need to show that 
\begin{equation}
\lim_{s\rightarrow +\infty }\pi _{y}\phi _{t^{\ast \ast },s}\left( p^{\ast
\ast }\right) =0.  \label{eq:convergence-positive-y}
\end{equation}%
Suppose that this is not the case. Then there would exist a $\delta $ and a
sequence $s_{i}\rightarrow +\infty $ such that 
\begin{equation}
\left\vert \pi _{y}\phi _{t^{\ast \ast },s_{i}}\left( p^{\ast \ast }\right)
\right\vert >\delta .  \label{eq:convergence-positive-y-contr}
\end{equation}%
Let us choose $T$ positive enough so that $V\left( x,y\right) =\left\Vert
y\right\Vert ^{2}$ would be a Lyapunov function on $[T,+\infty )\times D$,
for 
\begin{equation*}
D=\left[ x^{\ast \ast }-2\rho ,x^{\ast \ast }+2\rho \right] \times \left\{
\left\Vert y\right\Vert \geq \frac{\delta }{2}\right\} .
\end{equation*}%
This is possible, since we can apply Lemma \ref{lem:Lap-f} provided that
condition (\ref{eq:t-cond}) is satisfied. The (\ref{eq:t-cond}) will hold
for $r=\frac{\delta }{2}$, provided that $T<t$, for $T$ large enough. Since $%
V$ is a Lyapunov function, by going forwards in time $\phi _{t^{\ast \ast
},s}\left( p^{\ast \ast }\right) $ will enter $\left[ x^{\ast \ast }-2\rho
,x^{\ast \ast }+2\rho \right] \times \left\{ \left\Vert y\right\Vert <\frac{%
\delta }{2}\right\} $ and once it enters, it will remain there. This
contradicts (\ref{eq:convergence-positive-y-contr}), hence we have
established (\ref{eq:convergence-positive-y}).
\end{proof}

\begin{corollary}
Consider $t^{\ast \ast }=0$. If%
\begin{eqnarray*}
0 &<&\min \left( a_{22},a_{33}\right) -6c_{b}, \\
0 &<&\lambda -2c_{b},
\end{eqnarray*}%
and%
\begin{eqnarray*}
\rho &<&\frac{1}{2}\left[ r\left( \frac{1}{2c_{b}}\min \left(
a_{22},a_{33}\right) -1\right) -x^{\ast \ast }\right] , \\
\rho &>&\frac{c_{b}\left( x^{\ast \ast }+2r\right) }{\lambda -2c_{b}},
\end{eqnarray*}%
then assumptions of Lemma \ref{lem:forward-sep-val} are satisfied.
\end{corollary}

Lemmas \ref{lem:backward-sep-val}, \ref{lem:forward-sep-val} lead to the
following algorithm for establishing bounds for $x_{-}^{\ast }$ and $%
x_{+}^{\ast }$: \medskip

\noindent \textbf{Algorithm 2:}

\begin{enumerate}
\item Fix $r,\rho >0$ and choose $t^{\ast }<0$ small enough to satisfy
assumptions of Lemma \ref{lem:backward-sep-val}. The Lemma \ref%
{lem:backward-sep-val} then ensures that%
\begin{equation*}
x_{-}^{\ast }\in \left[ x^{\ast }-\rho ,x^{\ast }+\rho \right] .
\end{equation*}

\item By rigorous numerical integration, evaluate the bound for $\phi
_{t^{\ast },t}(U_{x^{\ast },0}).$ The $t>0$ needs to be chosen large enough
so that $t^{\ast \ast }=t^{\ast }+t$ satisfies assumptions of Lemma \ref%
{lem:forward-sep-val}. The $t$ also needs to be large enough so that $\phi
_{t^{\ast },t}(U_{x^{\ast },0})\subset U_{x^{\ast \ast },b}$, for some $b>0$.

\item By Lemma \ref{lem:forward-sep-val} $\lim_{t\rightarrow +\infty }\phi
_{t^{\ast },t}\left( p^{\ast }\right) \in U_{x^{\ast \ast },b+\rho }$. This
ensures that%
\begin{equation*}
x_{+}^{\ast }\in \left[ x^{\ast \ast }-b-\rho ,x^{\ast \ast }+b+\rho \right]
.
\end{equation*}

We then have the following bound:%
\begin{equation*}
\frac{x_{+}^{\ast }}{x_{-}^{\ast }}\in \left[ \frac{x^{\ast \ast }-b-\rho }{%
x^{\ast }+\rho },\frac{x^{\ast \ast }+b+\rho }{x^{\ast }-\rho }\right] .
\end{equation*}
\end{enumerate}

\subsection{Application to the Shimizu Morioka system \label%
{sec:CAP-separatrix-value}}

In this section we prove Theorem \ref{th:sep-val}. We will consider a matrix 
$P$ consisting of the eigenvectors of (\ref{eq:S-M-ode2-limit}), of the form%
\begin{equation}
P=\left( 
\begin{array}{lll}
1 & a_{0}+1 & 0 \\ 
-1 & 1 & 0 \\ 
0 & 0 & 1%
\end{array}%
\right) .  \label{eq:P-def}
\end{equation}%
Thus, $P$ is transition from the eigenvectors basis to the standard one. In
coordinates $\gamma $ given by $\eta =P\gamma $, the ODE (\ref{eq:S-M-ode2})
takes form%
\begin{equation}
\gamma ^{\prime }=\left( A+B\left( t\right) \right) \gamma ,
\label{eq:ODE-sep-val-gamma}
\end{equation}%
for 
\begin{eqnarray}
A &=&\left( 
\begin{array}{ccc}
0 & 0 & 0 \\ 
0 & -\left( a_{0}+2\right) & 0 \\ 
0 & 0 & -a_{0}%
\end{array}%
\right) .  \nonumber \\
B\left( t\right) &=&\frac{a_{0}+1}{a_{0}+2}\left( 
\begin{array}{ccc}
-Z_{0}\left( t\right) & Z_{0}\left( t\right) & -2\frac{X_{0}\left( t\right) 
}{a_{0}+1} \\ 
-Z_{0}\left( t\right) & Z_{0}\left( t\right) & -2\frac{X_{0}\left( t\right) 
}{a_{0}+1} \\ 
-\left( a_{0}+2\right) X_{0}\left( t\right) & \left( a_{0}+2\right)
X_{0}\left( t\right) & 0%
\end{array}%
\right) .  \label{eq:B-form}
\end{eqnarray}

Since $a_{0}\approx 1.72$ from the bound on $X_{0}\left( t\right)
,Y_{0}\left( t\right) ,Z_{0}(t)$ from Theorem \ref{th:homoclinic} by (\ref%
{eq:B-form}), we see that each coefficient $b_{ij}\left( t\right) $ of the
matrix $\mathcal{B}\left( t\right) $ is bounded by 
\begin{eqnarray*}
\left\vert b_{ij}\left( t\right) \right\vert &\leq &\left( a_{0}+1\right)
3.5e^{-\xi \left\vert t\right\vert }\left\Vert \left( X_{0}\left( 0\right)
,Y_{0}\left( 0\right) ,Z_{0}(0)\right) \right\Vert \qquad \text{for }t\leq 0,
\\
\left\vert b_{ij}\left( T+t\right) \right\vert &\leq &\left( a_{0}+1\right)
3.5e^{-\xi t}\left\Vert \left( X_{0}\left( T\right) ,Y_{0}\left( T\right)
,Z_{0}(T)\right) \right\Vert \qquad \text{for }t\geq 0.
\end{eqnarray*}%
This finishes establishing the needed ingredients for Algorithm 2.\textbf{\ }
\begin{figure}[tbp]
\begin{center}
\includegraphics[height=4.5cm]{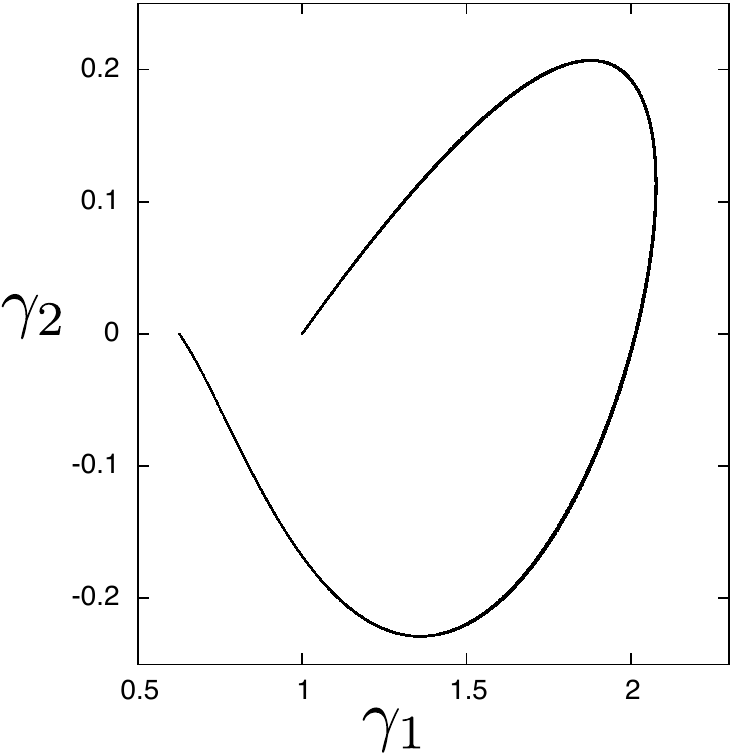}\qquad %
\includegraphics[height=4.5cm]{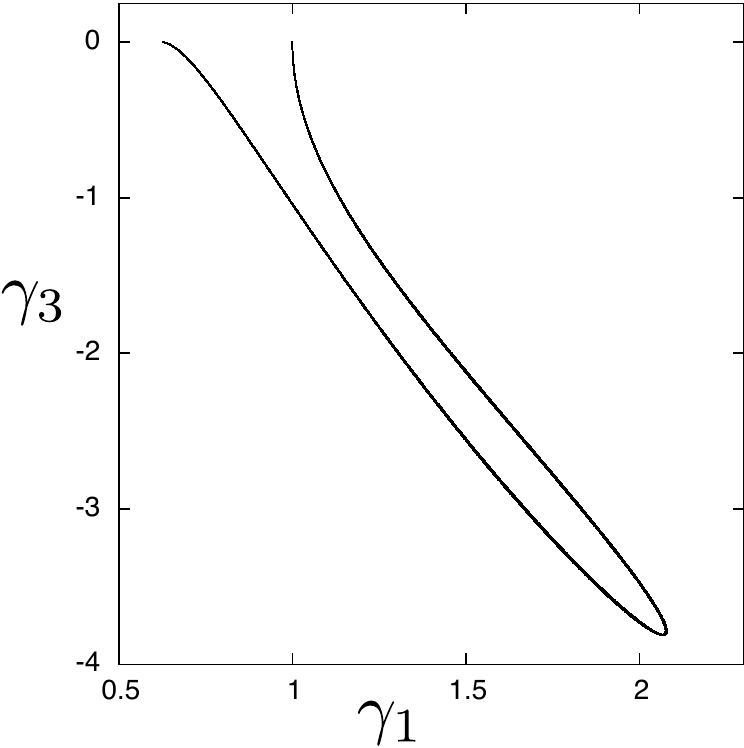}
\end{center}
\caption{The computer assisted bound on the heteroclinic trajectory of (%
\protect\ref{eq:ODE-sep-val-gamma}).}
\label{fig:gamma-trajectory}
\end{figure}

Below is the bound on the set $U_{1,\rho }$ from Lemma \ref%
{lem:backward-sep-val}, 
\begin{equation*}
U_{1,\rho }=\left( 
\begin{array}{c}
\lbrack 0.99984336210766,1.0001566378923] \\ 
\lbrack -0.00015663789234007,0.00015663789234007] \\ 
\lbrack -0.00015663789234007,0.00015663789234007]%
\end{array}%
\right) .
\end{equation*}%
By Lemma \ref{lem:backward-sep-val}, there is a point in $U_{1,0}$, that
converges to $(x_{-}^{\ast },0,0)\in U_{1,\rho }$. We therefore take the set 
$U_{1,0}$ as the initial point from which we integrate (\ref%
{eq:ODE-sep-val-gamma}) forward in time. A bound on all trajectories that
start in $U_{1,0}$ is depicted in Figure \ref{fig:gamma-trajectory}. (We
make also the same plot in coordinates $\eta $ in Figure \ref%
{fig:eta-trajectory}.) A trajectory from such enclosure makes a loop, to
finish at time $T$ closer to the origin, in a cubical enclosure $\Gamma
=\Gamma _{1}\times \Gamma _{2}\times \Gamma _{3}\subset \mathbb{R}^{3}$. We
use $\Gamma $ to compute the bound on the set $U_{\Gamma _{1},\rho }$ from
Lemma \ref{lem:forward-sep-val}, obtaining%
\begin{equation*}
U_{\Gamma _{1},\rho }\subset \left( 
\begin{array}{c}
\lbrack 0.62606812264791,0.62663392848044] \\ 
\lbrack -5.3960913395776e-05,5.3960913395776e-05] \\ 
\lbrack -5.3960913395776e-05,5.3960913395776e-05]%
\end{array}%
\right) .
\end{equation*}%
\begin{figure}[tbp]
\begin{center}
\includegraphics[height=4.5cm]{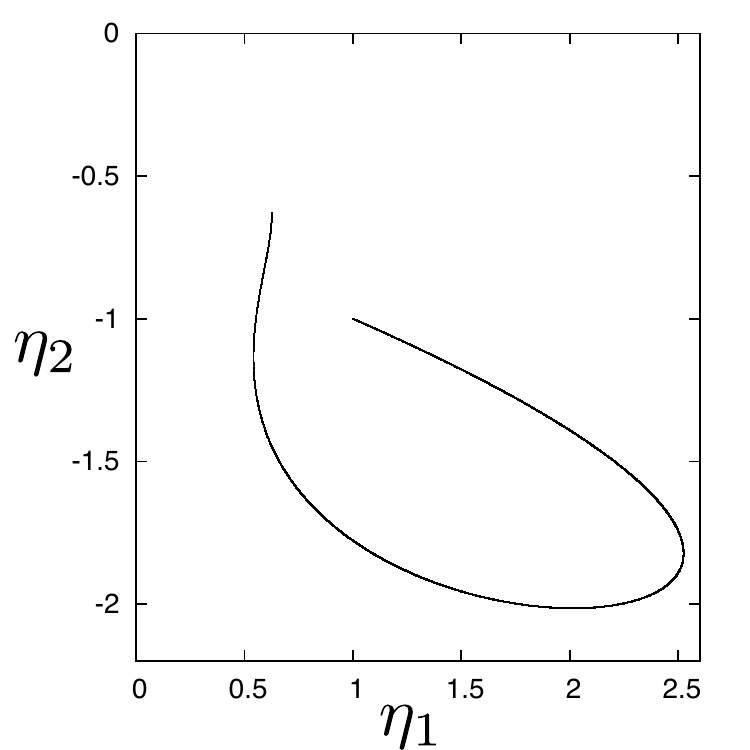}\qquad %
\includegraphics[height=4.5cm]{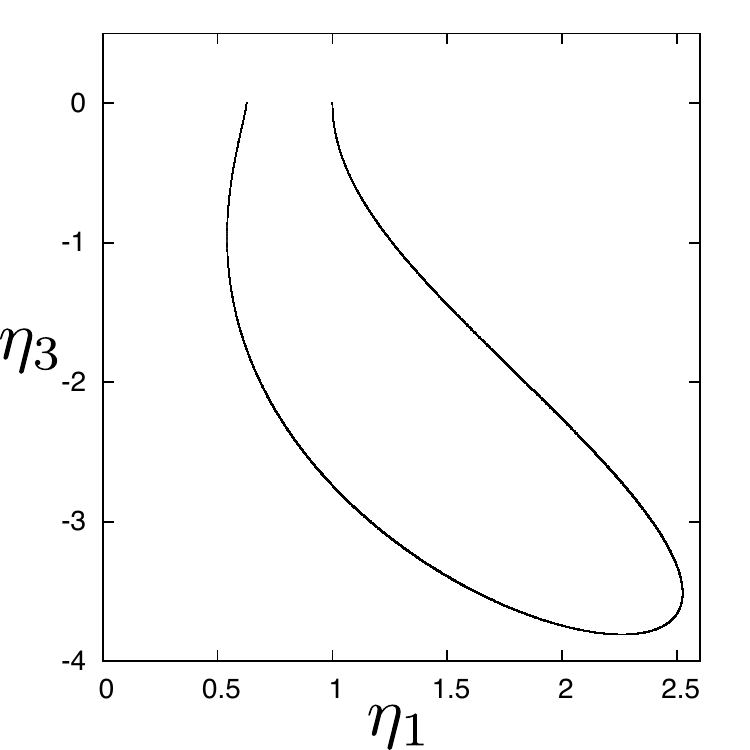}
\end{center}
\caption{The computer assisted bound on the heteroclinic trajectory of (%
\protect\ref{eq:S-M-ode2}).}
\label{fig:eta-trajectory}
\end{figure}
By Lemma \ref{lem:forward-sep-val}, there exists a point $(x_{+}^{\ast
},0,0)\in U_{\Gamma _{1},\rho }$ and a trajectory $\gamma (t)$ of (\ref%
{eq:ODE-sep-val-gamma}), for with 
\begin{eqnarray*}
\lim_{t\rightarrow -\infty }\gamma (t)& =(x_{-}^{\ast },0,0), \\
\lim_{t\rightarrow +\infty }\gamma (t)& =(x_{+}^{\ast },0,0).
\end{eqnarray*}%
From the fact that $\lim_{t\rightarrow \pm \infty }\eta \left( t\right)
=\lim_{t\rightarrow \pm \infty }P\gamma \left( t\right) $, we obtain the
claim of Theorem \ref{th:sep-val}.

The computer assisted proof has been done entirely by using the CAPD%
\footnote{%
computer assisted proofs in dynamics: http://capd.ii.uj.edu.pl/} package and
took under a second on a single core 3Ghz Intel i7 processor.

\section{Acknowledgements} We would like to thank the anonymous Reviewers, as well as the Editors for their comments, suggestions and corrections, which helped us improve our paper.

\section{Appendix}

\subsection{Proof of Theorem \protect\ref{th:Wu-cap-bound}\label%
{sec:Wu-cap-bound-proof}}

Before we give the proof, we introduce some auxiliary tools.

Let 
\begin{equation*}
J_{u}\left( z,L\right) =\left\{ \left( x,y\right) \in \mathbb{R}^{u}\times 
\mathbb{R}^{s}:\left\Vert \pi _{y}z-y\right\Vert \leq L\left\Vert \pi
_{x}z-y\right\Vert \right\} .
\end{equation*}%
The set $J_{u}\left( z,L\right) $ defines a cone of slope $L$, centered at $%
z $. Below theorem establish cone alignment for a map $\phi $, which
satisfies certain bounds on its derivative.

\begin{theorem}
\label{th:cone-alignment-maps}\cite[Theorem 27]{CZnhim} Let $U\subset 
\mathbb{R}^{u}\times \mathbb{R}^{s}$ be a convex neighborhood of zero and
assume that $\phi :U\rightarrow \mathbb{R}^{u}\times \mathbb{R}^{s}$ is a $%
C^{1}$. If for $M>0$%
\begin{eqnarray*}
v &\geq &\sup_{z\in U}\left\{ \left\Vert \frac{\partial \phi _{y}}{\partial y%
}\left( z\right) \right\Vert +\frac{1}{L}\left\Vert \frac{\partial \phi _{y}%
}{\partial x}\left( z\right) \right\Vert \right\} , \\
\zeta  &\leq &m\left[ \frac{\partial \phi _{x}}{\partial x}(U)\right]
-L\sup_{z\in U}\left\Vert \frac{\partial \phi _{x}}{\partial y}\left(
z\right) \right\Vert ,
\end{eqnarray*}%
and 
\begin{equation*}
\frac{\zeta }{v}>1,
\end{equation*}%
then for $z\in U$%
\begin{equation}
\phi \left( J_{u}\left( z,L\right) \cap U\right) \subset \mathrm{int}%
J_{u}\left( \phi \left( z\right) ,L\right) \cup \left\{ \phi \left( z\right)
\right\} .  \label{eq:cone-preservation}
\end{equation}
\end{theorem}

Note that if the vector field (\ref{eq:ode-wu-ws}) satisfies condition $%
\overrightarrow{\mu }<0<\overrightarrow{\xi }$, then by Theorem \ref%
{th:coeff-ode-map}, for sufficiently small $h>0,$ and for $\mu \left(
h\right) $ and $\xi \left( h\right) $ defined in (\ref{eq:mu-h}--\ref%
{eq:xi-h}), the 
\begin{equation*}
v=\mu \left( h\right) ,\qquad \qquad \zeta =\xi \left( h\right) ,
\end{equation*}%
will satisfy the assumptions of Theorem \ref{th:cone-alignment-maps} for $M=L
$. This will imply (\ref{eq:cone-preservation}) for $\phi =\Phi _{h}$, for
the flow $\Phi $ induced by (\ref{eq:ode-wu-ws}).

We now give the sketch of the proof of Theorem \ref{th:Wu-cap-bound}.

\begin{proof}[Proof of Theorem \protect\ref{th:Wu-cap-bound}]
The proof of the theorem follows from a mirror argument to the proof of
Theorem 30 from \cite{CZmelnikov}. The one important difference is that the
result from \cite{CZmelnikov} is written in the context where in addition to
the hyperbolic directions $x,y$, we also have a center coordinate. Here such
coordinate does not exist, which allows us to obtain better bounds on the
slope of the established manifold. Also, due to the lack of the center
coordinate, we have less inequalities in the assumptions of our theorem
compared to \cite{CZmelnikov}. Instead of repeating the proof of Theorem 30
from \cite{CZmelnikov} we refer the reader to the source, and will focus
here on the Lipschitz bounds of the manifold, which is the improvement of
the current result over \cite{CZmelnikov}.

The proof of Theorem 30 from \cite{CZmelnikov} follows from a graph
transform method \cite{CZnhim}. We start with a flat function 
\begin{equation}
\overline{B}_{u}\left( R\right) \ni x\rightarrow 0\in \overline{B}_{s}\left(
R\right) ,  \label{eq:initial-hor-disc}
\end{equation}%
and propagate it using the graph transform method. The manifold $W^{u}$ is
obtained by passing to the limit. Theorem 30 from \cite{CZmelnikov} ensures
that $W^{u}$ is a graph of the function $w^{u}$, meaning that we have (\ref%
{eq:Wu-as-graph}). (We refer the reader to \cite{CZmelnikov} for the proof
of this procedure that is just outline here.) We will focus here on the
Lipschitz bounds obtained at the limit. From \cite{CZmelnikov} we would
obtain directly the Lipschitz bound $1/L$, but here we will show that the
bound (due to lack of the center coordinate) is in fact $L$.

The important issue is that (\ref{eq:rate-cond}) ensures that the
assumptions of Theorem \ref{th:cone-alignment-maps} are satisfied (with $M=L$%
,) for the map $z\rightarrow \Phi _{h}\left( z\right) ,$ by taking $v=\mu
\left( h,L\right) $, $\zeta =\xi \left( h,L\right) $ and sufficiently small $%
h>0$. This means that the cones $J_{u}$ are preserved along the flow in the
sense of (\ref{eq:cone-preservation}); i.e. that 
\begin{equation*}
\Phi _{h}\left( J_{u}\left( z,M\right) \cap U\right) \subset \mathrm{int}%
J_{u}\left( \Phi _{h}\left( z\right) ,M\right) \cup \left\{ \Phi _{h}\left(
z\right) \right\} .
\end{equation*}%
For any $x\in \overline{B}_{u}\left( R\right) $ the graph of (\ref%
{eq:initial-hor-disc}) is inside of the cone $J_{u}\left( \left( x,0\right)
,L\right) $. The graph of (\ref{eq:initial-hor-disc}) after propagating by $%
\Phi _{h}$ using the graph transform, will also be contained in cones $J_{u}$%
. This implies that, after passing to the limit with the graph transform,
for any $x_{1},x_{2}\in \overline{B}_{u}\left( R\right) $%
\begin{equation*}
\left( x_{1},w^{u}\left( x_{1}\right) \right) \in J_{u}\left( \left(
x_{2},w^{u}\left( x_{2}\right) \right) ,L\right) ,
\end{equation*}%
hence%
\begin{equation*}
\left\Vert w^{u}\left( x_{2}\right) -w^{u}\left( x_{1}\right) \right\Vert
\leq L\left\Vert x_{1}-x_{2}\right\Vert .
\end{equation*}%
This means that $w^{u}$ is Lipschitz with constant $L$, as required.
\end{proof}

\bibliography{bibl}

\end{document}